\documentclass[11pt]{amsart}
\usepackage[mathscr]{eucal}
\usepackage{amsfonts, amsmath, amscd, color, latexsym, mathrsfs, enumerate, amssymb}
\counterwithin{figure}{section}

\usepackage{graphicx}
\usepackage{tikz}

\makeindex

\newtheorem{thm}{{{Theorem}}}[section]
\newtheorem{prop}[thm]{{Proposition}}
\newtheorem{lem}[thm]{{Lemma}}
\newtheorem{cor}[thm]{{Corollary}}

\newtheorem{exa}[thm]{{Example}}
\newtheorem{remark}[thm]{Remark}
\newtheorem{conj}[thm]{Conjecture}
\numberwithin{equation}{section}
\newtheorem{Def}[equation]{Definition}

\newcommand{\F}{\mathbb{F}}

\newcommand{\Gonetwo}{\mathcal{G}^{SS}_{1}(2,p)}

\renewcommand{\a}{\alpha}

\newcommand{\g}{\gamma}
\newcommand{\G}{\Gamma}
\renewcommand{\d}{\delta}
\newcommand{\D}{\Delta}

\renewcommand{\k}{\kappa}

\renewcommand{\l}{\lambda}
\renewcommand{\L}{\Lambda}
\newcommand{\m}{\mu}
\newcommand{\n}{\nu}

\renewcommand{\r}{\rho}
\newcommand{\s}{\sigma}
\renewcommand{\SS}{\Sigma}

\newcommand{\f}{\varphi}
\newcommand{\x}{\xi}

\newcommand{\z}{\zeta}
\newcommand{\p}{\psi}
\renewcommand{\P}{\Psi}

\newcommand{\Ac}{{\mathcal A}}
\newcommand{\Bc}{{\mathcal B}}
\newcommand{\Cc}{{\mathcal C}}

\newcommand{\Gc}{{\mathcal G}}
\newcommand{\Ic}{{\mathcal I}}
\newcommand{\Lc}{{\mathbb L}}
\newcommand{\Nc}{{\mathcal N}}
\newcommand{\Oc}{{\it O}}

\newcommand{\Hc}{{\mathcal H}}

\newcommand{\Sc}{{\mathcal S}}

\newcommand{\C}{{\mathbb C}}

\newcommand{\R}{{\mathbb R}}

\newcommand{\Z}{{\mathbb Z}}
\newcommand{\Q}{{\mathbb Q}}

\newcommand{\Fb}{{\mathbb F}}

\def\({\left(}
\def\){\right)}

\def\lbr{\langle}
\def\rbr{\rangle}

\def\l\{{\left\{}
\def\r\}{\right\}}
\def\wt{\widetilde}
\def\wh{\widehat}
\def\wbar{\overline}
\def\id{{\rm id}}

\def\gfr{\mathfrak{g}}
\def\sl{\mathfrak{sl}}
\def\kfr{\mathfrak{k}}
\def\pfr{\mathfrak{p}}
\def\afr{\mathfrak{a}}

\def\tr{{\rm tr}}
\def\ad{{\rm ad}}
\def\Ad{{\rm Ad}}

\def\Ker{{\rm Ker\,}}
\def\Sfr{\mathfrak{S}}
\def\FD{{\rm FD}}

\def\Ran{{\rm R}}
\def\Cay{{\rm Cay}}
\def\Leb{{\rm Leb}}
\def\ev{{\rm ev\,}}
\def\diam{{\rm diam\,}}
\def\Pb{{\bf P}}

\def\Eb{{\bf E}\,}
\def\Zc{{\mathcal Z}}
\def\1{{\bf 1}}
\def\mix{{\rm mix}}
\def\TV{{\rm TV}}
\def\Var{{\rm Var}}

\def\gr{{\rm gr}}

\def\bi-form{\langle \ast, \ast\rangle}
\def\tr^{\,^t}
\def\lab{{\bf lab}}
\def\ord{{\rm ord}}
\def\Pf{{\rm Pf}}
\def\Ver{{\rm Ver}}
\def\diag{{\rm diag}}
\def\Unif{{\rm Unif}}

\newcommand{\bQ}{\overline{\mathbb{Q}}}
\newcommand{\bQp}{\overline{\mathbb{Q}_p}}
\newcommand{\bZ}{\overline{\mathbb{Z}}}
\newcommand{\bZp}{\overline{\mathbb{Z}}_p}
\newcommand{\bF}{\overline{\mathbb{F}}}
\newcommand{\bFl}{\overline{\mathbb{F}}_\ell}
\newcommand{\bQl}{\overline{\mathbb{Q}_\ell}}
\newcommand{\Qp}{\mathbb{Q}_p}
\newcommand{\Ql}{\mathbb{Q}_\ell}
\newcommand{\Zp}{\mathbb{Z}_p}
\newcommand{\Zl}{\mathbb{Z}_\ell}
\newcommand{\oF}{\overline{F}}
\newcommand{\lra}{\longrightarrow}
\newcommand{\lla}{\longleftarrow}
\newcommand{\K}{\mathbb{K}}
\newcommand{\vp}{\varphi}
\newcommand{\A}{\mathbb{A}}
\newcommand{\mA}{\mathcal{A}}
\newcommand{\mL}{\mathcal{L}}
\newcommand{\mM}{\mathcal{M}}
\newcommand{\wA}{\widehat{A}}
\renewcommand{\O}{\mathcal{O}}
\renewcommand{\P}{\frak{P}}
\newcommand{\br}{\overline{\rho}}
\newcommand{\la}{\lambda}
\newcommand{\w}{\omega}
\newcommand{\brho}{\overline{\rho}}
\newcommand{\tG}{\widetilde{G}}
\newcommand{\wf}{\widetilde{f}}
\newcommand{\ds}{\displaystyle}
\newcommand{\bs}{\backslash}
\newcommand{\td}{\widetilde{\delta}}

\begin{document}

\title[Isogeny graphs on superspecial abelian varieties]
      {Isogeny graphs on superspecial abelian varieties: Eigenvalues and Connection to Bruhat-Tits buildings}

\author{Yusuke Aikawa}
\address{Yusuke Aikawa \\
Information Technology R\&D Center, Mitsubishi Electric Corporation JAPAN}
\email{Aikawa.Yusuke@bc.MitsubishiElectric.co.jp}

\author{Ryokichi Tanaka}
\address{Ryokichi Tanaka \\
Department of Mathematics,
Kyoto University, Kyoto 606-8502 JAPAN}
\email{rtanaka@math.kyoto-u.ac.jp}

\author{Takuya Yamauchi}
\address{Takuya Yamauchi \\
Mathematical Institute Tohoku University, Sendai 980-8578 JAPAN}
\email{takuya.yamauchi.c3@tohoku.ac.jp}

\subjclass[2010]{}

\date{\today}

\maketitle

\begin{abstract}
We study for each fixed integer $g \ge 2$, for all primes $\ell$ and $p$ with $\ell \neq p$,
finite regular directed graphs associated with 
the set of equivalence classes of $\ell$-marked principally polarized superspecial abelian varieties of dimension $g$ in 
characteristic $p$, and show that the adjacency matrices have real eigenvalues with spectral gaps independent of $p$.
This implies a rapid mixing property of natural random walks on the family of isogeny graphs beyond the elliptic curve case and suggests
a potential construction of     
the Charles-Goren-Lauter type cryptographic hash functions for abelian varieties. 
We give explicit lower bounds for the gaps in terms of 
the Kazhdan constant for the symplectic group when $g \ge 2$,
and discuss optimal values in view of the theory of 
automorphic representations when $g=2$. 
As a by-product, we also show that the finite regular directed graphs constructed by 
Jordan-Zaytman also has the same property. 
\end{abstract}

\tableofcontents 

\section{Introduction}\label{intro}
Isogeny graphs are finite graphs associated with elliptic curves, more generally, abelian varieties over finite fields.
They have attracted attention not only in arithmetic geometry but also in cryptography since the objects consist a building block in a prospective secure encryption scheme.
It is believed that finding a path between an arbitrary pair of points is highly intractable in those graphs whereas a relatively short random walk path ends up with a fairly randomized vertex.
In this paper, we study a random walk, thus mainly concerning the latter, on the isogeny graphs based on principally polarized superspecial abelian varieties over $\wbar \Fb_p$ of dimension $g$ at least $2$ formed by $(\ell)^g$-isogenies with $p \neq \ell$ for primes $p$ and $\ell$.
This is one of natural generalizations beyond the supersingular elliptic curves, the case corresponding to dimension $1$.

\subsection{Main Theorems}

To go into further explanation we need to fix some notation and 
the details are left to the relevant sections.
Let $p$ be a prime and $g$ be a positive integer. 
Fix an algebraically closed field $\bF_p$ of the finite field $\F_p=\Z/p\Z$. 
Let $SS_g(p)$ be the set of isomorphism classes of 
all principally polarized superspecial abelian varieties over $\bF_p$ which are of 
dimension $g$. We denote such an abelian variety $A$ endowed with the 
principal polarization $\mL$ which is an ample line bundle $\mL$ on $A$ with 
the trivial Euler-Poincar\'e characteristic. 
For a principally polarized superspecial abelian variety $(A,\mL)$ 
we write $[(A,\mL)]$ for the class of $(A,\mL)$ in $SS_g(p)$. 
It is known that $SS_g(p)$ is a finite set and 
more precisely that  
\[
C_1(g) p^{g(g+1)/2}\le |SS_g(p)|\le C_2(g) p^{g(g+1)/2}
\]
for all large enough $p$ and for some positive constants $C_1(g)$ and $C_2(g)$ depending only on $g$ (it follows from the mass formula (1.2) in p.1419 of \cite{Yu}). 

Fix a representative $(A_0,\mL_0)$ in a class of $SS_g(p)$ and 
a prime $\ell\neq p$.  
For each $(A,\mL)$ in a class of $SS_g(p)$, there exists an 
isogeny $\phi_A:A_0\lra A$ of $\ell$-power degree such that 
${\rm Ker}(\phi_A)$ is a maximal totally isotropic subspace of 
$A[\ell^n]$ for some $n\ge 0$ 
(it follows from Theorem \ref{JZ} in Section \ref{AFJZ} in this paper and Theorem 34 of \cite{JZ}). 
We call $\phi_A$ an $\ell$-marking of $(A,\mL)$ from $(A_0,\mL_0)$. 
Two $\ell$-markings of $(A,\mL)$ from $(A_0,\mL_0)$ differ by only an 
element in 
$$\G(A_0)^\dagger:=\{f\in ({\rm End}(A_0)\otimes_\Z \Z[1/\ell])^\times\ |\ f\circ f^\dagger=
f^\dagger\circ f\in \Z[1/\ell]^\times {\rm id}_{A_0}\}$$
where $\dagger$ stands for the Rosati involution associated to $\mL_0$ 
(see Proposition \ref{diff-markings}). 

We define 
\begin{equation}\label{pla0}
SS_g(p,\ell,A_0,\mL_0):=\{(A,\mL,\phi_A)\}/\sim 
\end{equation}
where $[(A,\mL)]\in SS_g(p)$ and $\phi_A$ is an 
$\ell$-marking from $(A_0,\mL_0)$.  
Here two objects $(A_1,\mL_1,\phi_{A_1})$ and $(A_2,\mL_2,\phi_{A_2})$ are
said to be equivalent if there exists an isomorphism $f:(A_1,\mL_1)\lra (A_2,\mL_2)$ such that 
$f\circ \phi_{A_1}$ and $\phi_{A_2}$ differ by only an element of 
$\G(A_0)^\dagger$ in which case we write 
$(A_1,\mL_1,\phi_{A_1})\sim(A_2,\mL_2,\phi_{A_2})$. 
We write $[(A,\mL,\phi_{A})]\in SS_g(p,\ell,A_0,\mL_0)$ for the class of 
$(A,\mL,\phi_A)$ where $(A,\mL)$ is a principally polarized superspecial abelian variety with an $\ell$-marking $\phi_A$ from $(A_0,\mL_0)$.  
Let $C$ be a maximal totally isotropic subgroup (or a Lagragian subspace in 
other words) of $A[\ell]$. Then the quotient 
$A_C=A/C$ yields an object, say $(A_C,\mL_C)$ in a class in $SS_g(p)$ 
 and the natural surjection $f_C:A\lra A_C$ is called an $(\ell)^g$-isogeny 
 (see Proposition \ref{uni-1} and Definition \ref{ellg-iso}). 
Any $(\ell)^g$-isogeny between two objects in $SS_g(p)$ arises in this way. 
We remark that the number of maximal totally isotropic subgroups $A[\ell]$ is 
$$N_g(\ell):=\ds\prod_{k=1}^g(\ell^k+1)$$
for each $A$.  

We are now ready to define the ($\ell$-marked) $(\ell)^g$-isogeny graph $\Gc^{SS}_{g}(\ell,p)$ for 
$SS_g(p,\ell,A_0,\mL_0)$ is defined as a directed graph where
\begin{itemize}
\item the set of vertices $V(\Gc^{SS}_{g}(\ell,p))$ is $SS_g(p,\ell,A_0,\mL_0)$ and
\item the set of directed edges between two vertices $v_1$ and $v_2$ is the set of equivalence classes of  $(\ell)^g$-isogenies between corresponding principally polarized superspecial abelian varieties commuting with marking isogenies representing 
$v_1$ and $v_2$. In other words, if $v_1$  and $v_2$ correspond to 
$[(A_1,\mL_1,\phi_{A_1})]$ and  $[(A_2,\mL_2,\phi_{A_2})]$ with $\ell$-markings 
$\phi_{A_1}:(A_0,\mL_0)\lra (A_1,\mL_1)$ and 
$\phi_{A_2}:(A_0,\mL_0)\lra (A_2,\mL_2)$ respectively, then 
an edge from $v_1$ to $v_2$ is an $(\ell)^g$-isogeny $f:(A_1,\mL_1)\lra (A_2,\mL_2)$ such that 
two markings $f\circ \phi_{A_1}$ and $\phi_{A_2}$ of $(A_2,\mL_2)$ from $(A_0,\mL_0)$ 
differ by an element of $\G(A_0)^\dagger$. 
\end{itemize}
Our graph is regular since it has $N_g(\ell)$-outgoing edges from each vertex, possibly loops and multiple edges from one to another.
The associated random walk operator for $\Gc^{SS}_{g}(\ell,p)$ is self-adjoint with respect to a weighted inner product by the inverse of the order of the reduced automorphism group (see Section \ref{Sec:rwop}).
Our first main result is the following:
\begin{thm}\label{main}Let $p$ be a prime. 
For each fixed integer $g \ge 2$ and for each fixed prime $\ell\neq p$,
the finite $N_g(\ell)$-regular directed multigraph $\Gc^{SS}_{g}(\ell,p)$ has the second largest eigenvalue of the normalized Laplacian satisfying that
\[
\lambda_2(\Gc^{SS}_{g}(\ell,p)) \ge c_{g, \ell}>0,
\]
where $c_{g, \ell}$ is a positive constant depending only on $g$ and $\ell$.
\end{thm}

We defined the normalized Laplacian $\D$ on a regular directed multigraph $\Gc$ of degree $d$ by $\D=1-(1/d)M$ for the adjacency matrix $M$ of $\Gc$.
Note that
$\D$ has the simple smallest eigenvalue $0$ provided that the graph is strongly connected, i.e., there exists a directed edge path from any vertex to any other vertex.
In Theorem \ref{main}, we actually have an explicit lower bound for $\lambda_2$:
For every integer $g \ge 2$, for all primes $\ell$ and $p$ with $p \neq \ell$,
\[
\lambda_2\(\Gc^{SS}_{g}(\ell,p)\) \ge \frac{1}{4(g+2)}\(\frac{\ell-1}{2(\ell-1)+3\sqrt{2\ell (\ell+1)}}\)^2,
\]
(Corollary \ref{Cor:Oh} in Section \ref{Sec:Oh}).
In the course of the proof of Theorem \ref{main}, we relate $\Gc^{SS}_{g}(\ell,p)$ 
to a finite quotient $\G\bs\mathcal{S}_g$ (see Subsection \ref{btq1}) of the special 1-complex $\mathcal{S}_g$ defined in terms of the Bruhat-Tits building for $PGSp_g(\Q_\ell)$ 
(see Theorem \ref{JZ} and Section \ref{Sec:special}). 
We then move on  $\mathcal{S}_g$ to prove the desired property by using 
Kazhdan's Property (T) of $PGSp_g(\Q_\ell)$ for $g \ge 2$. 

In \cite{JZ}, Jordan and Zaytman introduced a {\it big isogeny graph} $Gr_g(\ell, p)$ based on $SS_g(p)$. 
We will show in Section \ref{SSAV} and Section \ref{Comparison} that 
there exist natural identifications 
$$SS_g(p)\stackrel{1:1}{\lla}SS_g(p,\ell,A_0,\mL_0)\stackrel{1:1}{\lra}\Gamma\backslash\mathcal{S}_g$$
which induce natural isomorphisms as graphs between three objects:
\begin{enumerate}
\item $Gr_g(\ell, p)$,
\item $\Gc^{SS}_{g}(\ell,p)$, and 
\item the regular directed graph defined by $\Gamma\backslash\mathcal{S}_g$.
\end{enumerate}
It follows from this that the adjacency matrices of the above three graphs 
agree with each other. 
Therefore, the structure of Jordan-Zaytman's graph $Gr_g(\ell, p)$ is revealed by our main theorem:
\begin{thm}\label{mainJZ}
Let $p$ be a prime. 
For each fixed integer $g \ge 2$ and for each fixed prime $\ell\neq p$,
a finite $N_g(\ell)$-regular directed multigraph $Gr_g(\ell, p)$ has the same 
property as in Theorem \ref{main}. 
\end{thm}
This result implies the rapid mixing property of a lazy version of the walk; see \cite[Theorem 4.9]{FS2}. 

We discuss some theoretical features for each of our work and previous works due to Pizer and Jordan-Zaytman. Instead of using $Gr_1(\ell,p)$, 
Pizer handled the moduli space of supersingular elliptic curves with 
non-trivial levels to avoid that non-trivial automorphisms happen 
(see \cite{Pizer-Ram}, \cite{Pizer-graph}). Therefore, his graphs are regular 
undirected graphs so that they are Ramanujan by Eichler's theorem via 
Jacquet-Langlands theory. However, if $p\equiv 1$ mod 12, then 
each vertex of $Gr_1(\ell,p)$ does not have non-trivial automorphisms other than $-1$. 

Jordan-Zaytman's graphs $Gr_g(\ell,p)$ are useful and fit into the computational 
implementations (cf. \cite{CDS}, \cite{KT}, \cite{FS1}, \cite{FS2})
as explained in the next subsection.
However, it may be hard to 
directly obtain the uniform estimation of the eigenvalues of the normalized Laplacian. 
Our graphs do not, unfortunately, well-behave in the computational 
aspects. However, there is a natural correspondence between $SS_g(p,\ell,A_0,\mL_0)$ 
and $\mathcal{S}_g$ as explained. A point here is that these two objects have markings 
from a fixed object while $SS_g(p)$ does not have it. However, fortunately, 
there is a natural correspondence between $SS_g(p)$ and $SS_g(p,\ell,A_0,\mL_0)$. 
Then eventually we can relate $SS_g(p)$ with $\mathcal{S}_g$ via the intermediate 
object $SS_g(p,\ell,A_0,\mL_0)$.  

It seems interesting to consider the moduli space of principal 
polarized superspecial abelian varieties with a non-trivial level so that 
the reduced automorphism group of any object is trivial. 
This will be discussed somewhere else. 

\subsection{Motivation from isogeny-based cryptography}

This study is largely motivated by a possible approach to cryptographic hash functions from isogeny graphs.
Let us begin with a brief review of the notation of hash functions, which are widely used in computer science.
For a general reference on hash functions in cryptographic context, see Chapter 5 of \cite{KL} for example.

A hash function $H$ is an efficiently computable function taking as input a message of any length and outputting a value of fixed length $s$, i.e. $H: \{0,1\}^*\to \{0,1\}^{s}$, where $\{0,1\}^*=\cup_{n=1}^{\infty}\{0,1\}^n$.
A standard condition required for hash functions in cryptography is collision resistance; it is computationally hard for any probabilistic polynomial-time algorithm to find a pair of distinct messages $(m_1,m_2)$ such that $H(m_1)=H(m_2)$.
Collision resistant hash functions have numerous applications in cryptography.
For example, such functions are used as components of pseudo-random generators, Hash-based Message Authentication Code, digital signatures and so on.

However, despite its importance, it is hard to construct such a function because to design a collision resistant hash function requires suitable mixing and compressing bit strings of any length. 
As one can see in \cite{Goldreich}, \cite{CGL}, there is an approach to design hash functions by employing expander graphs on which random walks mix rapidly.
Due to Pizer's work \cite{Pizer-graph}, \cite{Pizer-Ram}, isogeny graphs of supersingular elliptic curves have attracted attention as a tool for realizing a good  expansion property.
In this subsection, we explain this research direction and state our research question.

\subsubsection*{CGL hash functions}\label{CGL}

Let $p$ and $\ell$ are distinct prime numbers.
Moreover, we impose $p\equiv 1\bmod 12$.
In thi case, the vertices on $Gr_1(\ell,p)$ have no automorphism other than $\pm1$.
Charles, Goren and Lauter \cite{CGL} proposed construction of hash functions based on 
$Gr_1(\ell,p)$.

We explaine the recipe of their construction of a hash function from the graph $Gr_1(2,p)$ as follows.
Let $E_0:y^2=f(x)$ be a fixed curve in $ SS_1(p)$ where $f(x)$ is a monic cubic, and $(E_0,E_{-1})$ be a fixed edge. 
We remark that the edges are undirected due to the existence of the dual isogeny.
The non-trivial 2-torsion points are points $P^0_i=(x_i,0)$ where $x_i$  are roots of the cubic $f(x)$ for $i=0,1,2$.
Subgroups generated by each $P_i$ lead 2-isogenies outgoing from $E_0$.
Here, the points are numbered by some order of $\F_{p^2}$ and we suppose that the edge $(E_{-1},E_0)$ corresponds to the subgroup $\langle P_2\rangle$. 
Let $m=(m_{n-1},\dots,m_0)\in \{0,1\}^n$ be a random $n$-bit message.
The message $m$ determines a non-backtracking walk
\footnote{To avoid trivial collision, we impose the condition of non-backtrack on walks in this construction.}
 on $Gr_1(2,p)$ in the following way.

First, we compute an isogeny $\phi_0:E_0\to E_1$ with kernel $\langle P^0_{m_0}\rangle$ by using V\'elu's formula \cite{V}.
Second, we have non-trivial three 2-torsion points on $E_1$ and we number one of them corresponding the dual of $\phi_0$ with $P^1_2$.
The remaining two 2-torsion points are numbered by the order in $\F_{p^2}$; $P^1_0$ and $P^1_1$. 
Then, we do a similar procedure for $E_1$ and obtain $E_2$ by computing the isogeny with kernel generated by $P^0_{m_1}$.
Finally, by repeating this computation, the end-point $E_n\in SS_1(p)$ is obtained as the terminal of the sequence of supersingular elliptic curves $(E_0,E_1,\dots,E_{n-1})$ such that $j_{E_{i-1}}\neq j_{E_{i}}$ for $i=1,\dots, n-1$.
To get a compressed value of $m$ from $E_n$, in \cite{CGL}, the authors propose using some linear function $f:\F_{p^2}\to \F_p$; that is, $H(m):=f(j_{E_n})$. 
In this way, we construct the function $H:\{0,1\}^*\to \{0,1\}^{\lfloor{\rm log}_2(p)\rfloor+1}$ from non-backtracking random walks on $\Gonetwo$, which is called CGL hash function now.
In a similar fashion, the hash function using 3-isogeny is also investigated in \cite{TTT} .

The Ramanujan property of $Gr_1(2,p)$ for $p\equiv 1\bmod 12$ guarantees efficient mixing processing of these functions
(for most precise results, see \cite[Theorems 1 and 3.5]{LubetzkyPeres}).
In view of security of these functions, the collision resistant property is supported by assumptions on hardness of computing a chain of isogenies between given isogenous supersingular elliptic curves.
Indeed, finding collisions yield to pairs of supersingular elliptic curves $(E,E')$ and two chains of $\ell$-isogenies between them whose kernels are distinct each other.

\subsubsection*{Isogeny-based Cryptography}
We provide a little bit about the recent progress in public-key cryptography using supersingular isogenies.
The above construction of the cryptographic hash function from supersingular isogenies opens the door to a new research area of practical public key cryptography
whose security relies on computational hardness of computing isogenies between given two supersingular elliptic curves.
Public key cryptography in such style is called isogeny-based cryptography now.


Here, what is important is that there is currently no known polynomial-time (even quantum)
algorithm to compute an isogeny between given two supersingular elliptic curves unlike the integer factorization problem or the discrete logarithm problem.
Indeed, an isogeny-based cryptography is regarded as an important object in the context of post-quantum cryptography:
it has been proposed as cryptographic primitives,
for example, SIDH (Supersingular Isogeny Diffie-Hellman)\cite{DJP, JD} and CSIDH (Commutative Supersingular Isogeny Diffie-Hellman)\cite{CLMPR}.
Therefore, isogeny-based cryptography is one of the promising candidates of post-quantum cryptography among lattice-based cryptography, code-based cryptography and multivariate cryptography.

\subsubsection*{Toward higher dimensional analogue of CGL hash functions}

There have been several studies on the big isogeny graph $Gr_g(\ell,p)$ defined in \cite{JZ} from both number theoretic and cryptographic viewpoints.
In the rest of this section, we describe recent progress in studies on the graphs $Gr_g(\ell,p)$ 
and a contribution of our work in this context.

Concerning the two dimensional case, the CGL-like construction of hash functions was first attempted by Takashima \cite{T},
which used the supersingular $(2)^2$-isogeny graph ({\it i.e.}, the case when $g=2$ and $\ell=2$).
However, Flynn and Ti \cite{FT} showed that this graph has many short cycles from which trivial collisions of random walks may be derived.
After these works, Castryck, Decru and Smith \cite{CDS} modified Takashima's construction and suggested to use a subgraph of isogeny graphs $Gr_2(2,p)$ of superspecial abelian varieties consisting of Jacobians of curves of genus $2$.
The idea here is to keep choosing paths to become good extension, which allow us to avoid trivial collisions. 
Moreover, they counted the number of good extensions of a $(2)^2$-isogeny (see Proposition 3 in \cite{CDS}).
There are eight good extensions for an isogeny between Jacobians,  which are suitable for associating 3-bit information to one step of a random walk.

In the case of abelian varieties of dimension $\geq 2$, the existence of nontrivial automorphisms complicates the structure of graphs.
For $g=2$, the classification of possible automorphism groups arising from Jacobians and elliptic product was done by Ibukiyama, Katsura and Oort \cite{IKO}. 
Based on these results, in the case when $g=2$ and $\ell=2$, Katsura and Takashima \cite{KT} counted the number of Richelot isogenies and decomposed Richelot isogenies up to isomorphism outgoing from Jacobians and those outgoing from elliptic products and computed the multiplicity of each edge.
Moreover, advancing this work further, Florit and Smith \cite{FS1} studied the local neighborhoods  of vertices and edges in  $Gr_2(2,p)$ and gave many illustrations.
In \cite{FS2}, they also investigated behavior of random walks on the big isogeny graphs and gave numerical experiments of the mixing rate of $Gr_2(2,p)$.

However, we know little about expansion properties of these graphs so far.
Our contribution is to give an affirmative answer to this question in Theorem \ref{main} and Theorem \ref{mainJZ}.
In this paper, good mixing property of the big isogeny $Gr_g(\ell,p)$ is shown as a result of proving that the isogeny graphs $\Gc^{SS}_{g}(\ell,p)$ defined in this paper have good expansion property and they are equivalent to the big isogeny graphs $Gr_g(\ell,p)$.
So, random walks on the graphs $\Gc^{SS}_{g}(\ell,p)$ and $Gr_g(\ell,p)$ tend to the natural stationary distribution rapidly.
This gives an evidence that the big isogeny graphs $Gr_g(\ell,p)$, which have been investigated, may be suitable for cnstruction of cryptographic hash functions from superspecial abelian varieties.

Finally, we give an example of an illustration of a graph considered in this paper, i.e. $\Gc^{SS}_{g}(\ell,p)$, which is equivalent to the big isogeny graph $Gr_g(\ell,p)$.
For $\ell=2$ and $p=13$, the graph $Gr_g(\ell,p)$ is computed in \cite{CDS} and \cite{KT} as illustrated below. 
\begin{figure}[h]
\includegraphics[scale=0.25]{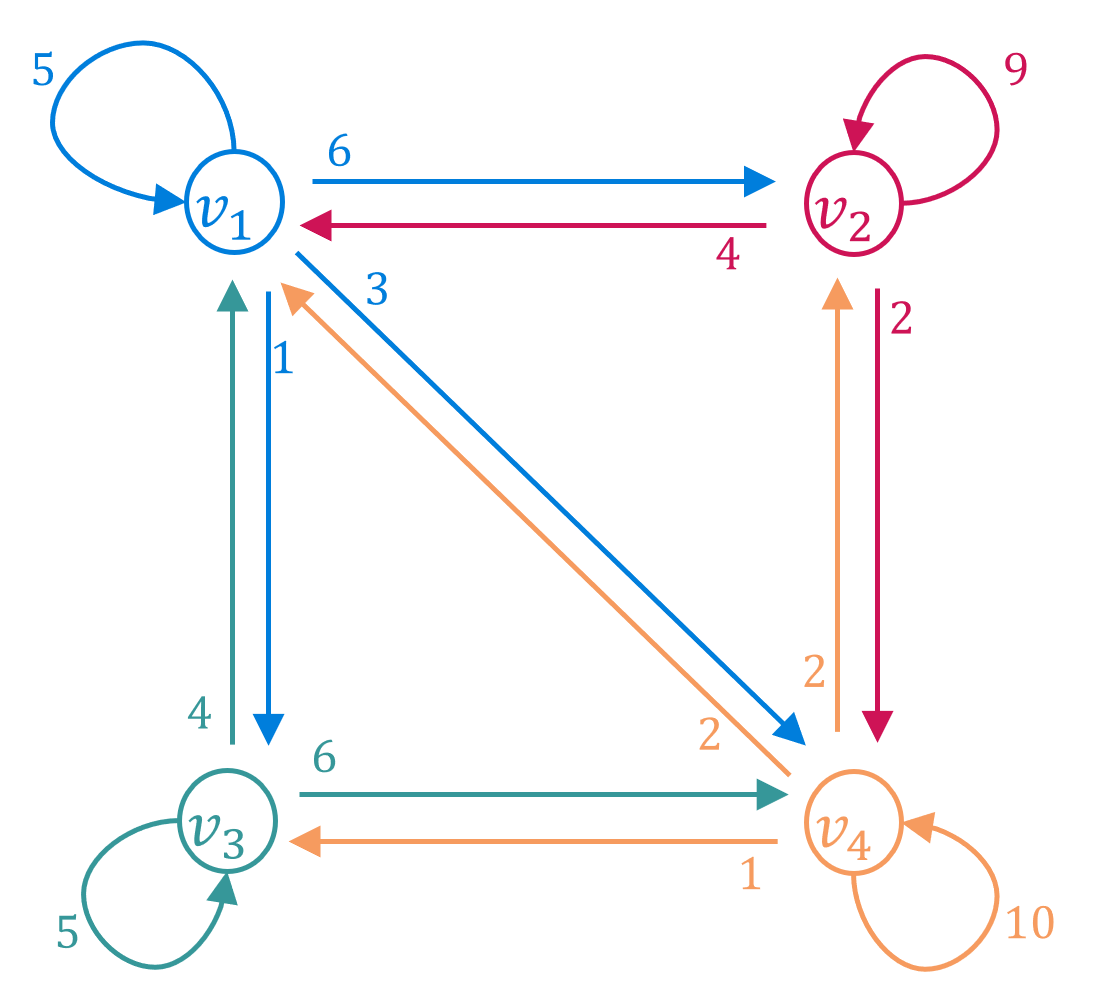}
\caption{An illustration of $Gr_2(2,13)$. The vertices $v_1$, $v_2$ and $v_3$ donote the Jacobians of curves defined by $C_1:y^2=(x^3-1)(x^3+4-\sqrt{2})$, $C_2:y^2=x(x^2-1)(x^2+5+2\sqrt{6})$ and $C_3:y^2=x^5-x$, respectively.
The vertex $v_4$ denotes the product of supersingular elliptic curves $y^2=(x-1)(x-3+2\sqrt{2})$.
The number on the side of a directed edge denotes the multiplicity of each edge.
For a more detailed illustration, see \S 7.1 of \cite{KT}.}
\label{Fig:p13}
\end{figure}

\subsection{Organization of this paper}
In Section \ref{SSAV}, we give two interpretations of $SS_g(p)$ according to 
works of Ibukiyama-Katsura-Oort-Serre and Jordan-Zaytman. 
The former is helpful to compute the cardinality of $SS_g(p)$ while the latter is 
helpful to make the compatibility of Hecke operators at $\ell$ transparent. 
As mentioned before, this is a crucial step to apply Property (T) (hence, Theorem 
\ref{thm:gap}) with our family $\{\mathcal{G}^{SS}_{g}(\ell, p)\}_{p\neq \ell}$. 
In Section \ref{Comparison},
we discuss
a comparison between the graph $\Gc_g^{SS}(\ell, p)$ and that of Jordan-Zaytman $Gr_g(\ell, p)$.
In Section \ref{BT}, we study Bruhat-Tits buildings 
for symplectic groups. Then, in Section \ref{PT}, the main result is proved in terms of 
the terminology in the precedent sections. Finally, in Section \ref{GrossDef} we give a speculation in view 
of the theory of automorphic forms.

\subsection{Notations}\label{notation}
For a set $X$, the cardinality is denoted by $|X|$.
Throughout the paper, we use the Landau asymptotic notations: for positive real-valued functions $f(n)$ and $g(n)$ for integers $n$,
we denote by $f(n)=o(g(n))$ if $f(n)/g(n) \to 0$ as $n \to \infty$,
by $f(n)=O(g(n))$ if there exists a positive constant $C>0$ such that $f(n) \le C g(n)$ for all large enough $n$,
and $f(n)=\Theta(g(n))$ if we have both $f(n)=O(g(n))$ and $g(n)=O(f(n))$. 

Let $n$ be a positive integer and $I_n$ the identity matrix of size $n$. 
Put 
\[
J_n=
\begin{pmatrix}
0 & I_n \\
-I_n & 0
\end{pmatrix}
.
\]  
We define a functor 
$G:(Rings)\lra (Sets)$ from the category of rings to the category of sets by 
$$G(R):=\{M\in M_{2n}(R)\ |\ {}^t MJ_nM=\nu(M)M,\ \text{for some }\nu(M)\in R^\times\}
$$
for each commutative ring $R$ and we call $\nu(M)$ the similitude of $M$. 
It is well-known that the functor $G$ is represented 
by a smooth group scheme $GSp_n$ over $\Z$. The similitude defines a homomorphism 
$\nu:GSp_n\lra GL_1$ as group schemes over $\Z$. We define  
$Sp_n:={\rm Ker}(\nu)$ which is called the symplectic group of rank $n$. 
The similitude splits and in fact it is given by $a\mapsto \diag(I_n,a I_n)$. 
It follows from this that $G\simeq Sp_n \rtimes GL_1$. 
In the sections related to abelian varieties, we put $n=g$ while we keep 
$n$ in Section \ref{BT} through \ref{PT}. 

For any algebraic group $H$ over  a field, we denote by $Z_{H}$ 
the center of $H$.

\subsection*{Acknowledgements}
We would like to thank Professors Yevgeny Zaytman and Bruce Jordan for informing us of issues in the first version of this paper and correcting references; the current version has been revised substantially since then-----we greatly acknowledge them for their comments.
We would also like to thank
 Professor Tsuyoshi Takagi for many helpful discussions and encouragement, and Professor Ken-ichi Kawarabayashi for fostering an ideal environment which made this collaboration possible.
Y.A.\ is supported by JST, ACT-X Grant Number JPMJAX2001, Japan. 
R.T.\ is partially supported by 
JSPS Grant-in-Aid for Scientific Research JP20K03602 and JST, ACT-X Grant Number JPMJAX190J, Japan. 
T.Y.\ is partially supported by 
JSPS KAKENHI Grant Number (B) No.19H01778.

\section{Superspecial abelian varieties}\label{SSAV}
In this section we refer \cite{Mum} for some general facts of 
abelian varieties. The purpose here is to understand 
Theorem 2.10 of \cite{IKO} in terms of the adelic language which is implicitly given there. 
Another formulation is also given in terms of $\ell$-adic Tate modules 
(see also Theorem 46 of \cite{JZ} in more general setting). 
This explains the compatibility of Hecke operators on principally polarized superspecial abelian varities 
and the special 1-complex of the Burhat-Tits building in question. 
This result will be plugged into the main result in Section \ref{PT} to 
prove Theorem \ref{main}. 

\subsection{Superspecial abelian varieties}\label{ssav}
Let $p$ be a prime number and $k$ be an algebraically closed field of a finite field of characteristic $p$. In our purpose, we may put $k=\bF_p$. Let $A$ be an abelian variety over $k$ of dimension $g>0$ and we denote by $\wA={\rm Pic}^0(A)$ the dual abelian variety (cf. Section 9 of \cite{Milne}). 
The abelian variety $A$ is said to be superspecial if $A$ is isomorphic to 
$E^g=\overbrace{E\times \cdots \times E}^{g}$ for some supersingular elliptic curve $E$ over $k$ (see Sections 1.6 and 1.7 of \cite{LO} for another definition 
in terms of $a$-number).  
As explained in loc.cit., for any fixed supersingular elliptic curve $E_0$ over $k$, 
every superspecial abelian variety of dimension $g\ge 2$ is isomorphic to $E^g_0$.
(Here the assumption $g\ge 2$ is essential and indeed, this is not true for $g=1$. 
See also Theorem 4.1 in Chapter V of \cite{Sil}.)
Throughout this section, we fix a supersingular elliptic curve  $E_0$. 
\subsection{Principal polarizations}
Let $A$ be an abelian variety over $k=\bF_p$. A polarization is a class of 
the N\'eron-Severi group ${\rm NS}(A):={\rm Pic}(A)/{\rm Pic}^0(A)$ which is represented by 
an ample line bundle on $A$. The definition of polarizations here is 
different from the usual one but it is equivalent by Remark 13.2 of \cite{Milne} since $k=\bF_p$. 

For each ample line bundle $\mL$ we define a 
homomorphism 
$$\phi_\mL:A\lra \wA,\ x\mapsto t^\ast_x(\mL)\otimes \mL^{-1}$$
where $t_x$ stands for the translation by $x$ and we denote by $t^\ast_x$ its pullback. 
By APPLICATION 1, p.60 in Section 6 of Chapter II in \cite{Mum}, $\phi_\mL$ is an isogeny, hence 
it has a finite kernel since $\mL$ is ample. If we write $D$ for an ample divisor on $A$ corresponding to $\mL$, namely, $\mL\simeq \O(D)$, then by Riemann-Roch  theorem in p.150 of loc.cit., we have 
\begin{equation}\label{rrt}
\chi(\mL)=\frac{(D^g)}{g!},\ \chi(\mL)^2=\deg \phi_\mL
\end{equation}
where $\chi(\mL)$ stands for the Euler-Poincare characteristic of $\mL$ and 
$(D^g)$ is the $g$-fold self-intersection number of $D$. Since $\mL$ is ample, 
$\chi(\mL)>0$. 

\begin{Def}\label{pp}Keep the notation being as above. 
\begin{enumerate}
\item  An ample line bundle $\mL\simeq \O(D)$ on $A$ is said to be 
a principal polarization if $\deg \phi_\mL=1$, equivalently,  $(D^g)=g!$. 
\item For a principal polarization $\mL$ on $A$, we call a couple $(A,\mL)$ 
a principally polarized abelian variety. For two polarized abelian varieties 
$(A_1,\mL_1),\ (A_2,\mL_2)$, a morphism between them is a homomorphism 
$f:A_1\lra A_2$ such that $\phi_{\mL_1}=\widehat{f}\circ \phi_{\mL_2}\circ f$  
where $\widehat{f}$ is the dual of $f$ defined by the pullback of $f$ on line bundles. 
Since $\widehat{f}\circ \phi_{\mL_2}\circ f=\phi_{f^\ast\mL_2}$, 
the above condition is equivalent to $\phi_{f^\ast\mL_2}=\phi_{\mL_1}$. 
\end{enumerate}
\end{Def}

\begin{prop}\label{uni-1} Let $(A,\mL)$ be a principally polarized abelian 
variety in characteristic $p$. Let $\ell$ be a prime number different from $p$ and 
$C$ be a maximal totally isotropic subspace of $A[\ell^n]$ for $n\in \Z_{\ge 0}$ with respect to the 
Weil pairing associated to $\mL$. 
Then, there exists an ample line bundle $\mL_C$ on the quotient abelian variety $A_C:=A/C$  which is unique up to isomoprhism 
such that $(A_C,\mL_{A_C})$ is  a principally polarized abelian 
variety in characteristic $p$ such that $f^\ast_C \mL_{A_C}=\mL^{\otimes \ell^n}$ 
where $f_C:A\lra A_C$ is the natural surjection.  
\end{prop}
\begin{proof}Notice that $\mL$ is symmetric. The claim follows from 
(11.25) Proposition of \cite{vdGM}. 
\end{proof}
\begin{Def}\label{ellg-iso}Let $(A_1,\mL_1)$ and $(A_2,\mL_2)$ be two principally polarized abelian varieties in characteristic $p$. Let $\ell$ be  a prime different from $p$.  
\begin{enumerate}
\item An isogeny $f:A_1\lra A_2$ is said to be an $(\ell)^g$-isogeny if 
\begin{itemize}
\item ${\rm Ker}(f)$ is a maximal totally isotropic subspace of $A[\ell]$ 
with respect to the 
Weil pairing associated to $\mL_1$, and
\item $f^\ast\mL_2\simeq \mL^{\otimes \ell}_1$.
\end{itemize}
\item An isogeny $f:A_1\lra A_2$ is said to be an $\ell$-marking of $(A_2,\mL_2)$ from 
$(A_1,\mL_1)$ if 
$f^\ast \mL_2=\mL^{\otimes\ell^m}_1$ for some integer $m\ge 0$. 
\end{enumerate}
\end{Def}

\begin{prop}\label{dual-marking}Keep the notation in Definition \ref{ellg-iso}.  
Let $f:A_1\lra A_2$ be an $\ell$-marking of $(A_2,\mL_2)$ from 
$(A_1,\mL_1)$, then there exists an $\ell$-marking $\widetilde{f}:A_2\lra A_1$ 
of $(A_1,\mL_1)$ from 
$(A_2,\mL_2)$ such that 
$f\circ \widetilde{f}=[\ell^m]_{A_2}$ and  $\widetilde{f}\circ f=[\ell^m]_{A_1}$ for some integer $m\ge 0$.
\end{prop}
\begin{proof}
By Theorem 34 of \cite{JZ}, we may assume $f$ is an $(\ell)^g$-isogeny. 
Put $C={\rm Ker} f$. Then $(A_2,\mL_2)=(A_{1,C},\mL_{A_{1,C}})$ where $A_{1,C}=A_1/C$.  
It is easy to see that $D:=A_1[\ell]/C$ is a maximal totally isotropic subspace of 
$A_{1,C}[\ell]$ with respect to the Weil pairing associated to $\mL_{A_{1,C}}$. 
Therefore, we have an $(\ell)^g$-isogeny $\widetilde{f}:A_2\lra A_{1,C}/D$. 
However, $A_{1,C}/D=A/A[\ell]\simeq A$ and the later isomorphism induces 
the identification of $(A_{1,C}/D,\mL_D)$ and $(A_1,\mL_1)$ where 
$\mL_D$ stands for a unique descend of $\mL_{A_{1,C}}$ on $A_{1,C}/D$ 
(see Proposition \ref{uni-1}). 
The proportion of $f$ and $\widetilde{f}$ is symmetric and hence we have the claim.   
\end{proof}
We study the difference of two $\ell$-markings. 
Let us keep the notation in Definition \ref{ellg-iso}. 
By using the principal polarization $\mL_1$ we define the Rosati-involution $\dagger$ 
on ${\rm End}(A_1)$ by 
\begin{equation}\label{rosati}
f^\dagger=\phi^{-1}_{\mL_1}\circ \widehat{f}\circ \phi_{\mL_1},\ f\in {\rm End}(A_1). 
\end{equation}
Notice that $\dagger$ is an anti-involution.

\begin{prop}\label{diff-markings}Let us still keep the notation in Definition \ref{ellg-iso}.  
Let $f,h:A_1\lra A_2$ be two $\ell$-markings. 
Then there exists $\psi\in {\rm End}(A_1)\otimes \Z[1/\ell]$ such that 
$f\circ \psi=h$ and $\psi\circ \psi^{\dagger}= \psi^{\dagger}\circ \psi=[\ell^m]_{A_1}$ 
for some integer $m$. 
\end{prop}
\begin{proof}
For $f$, let $\widetilde{f}:A_2\lra A_1$ be an $(\ell)^g$-isogeny in Proposition \ref{dual-marking}. Put $\psi_1=\widetilde{f}\circ h\in {\rm End}(A_1)$. 
Then we have, by definition, 
$$\psi_1\circ \psi^{\dagger}_1=(\widetilde{f}\circ h) \circ 
(\phi^{-1}_{\mL_1}\circ \widehat{h}\circ \widehat{\widetilde{f}} \circ \phi_{\mL_1}).$$
It follows from Theorem 34 of \cite{JZ} and Definition \ref{ellg-iso} that 
$\widehat{\widetilde{f}}\circ \phi_{\mL_1}\circ \widetilde{f}=
\phi_{\mL^{\otimes\ell^{m}}_2}=\ell^m \phi_{\mL_2}$ and 
$\widehat{h}\circ \phi_{\mL_2}\circ h=\phi_{\mL^{\otimes\ell^{m'}}_1}=\ell^{m'} \phi_{\mL_1}$ 
for some integers $m',m\ge 0$. This yields  
$$\psi_1\circ \psi^{\dagger}_1=\ell^m \widetilde{f}\circ \phi^{-1}_{\mL_2}
\circ \widehat{\widetilde{f}} \circ \phi_{\mL_1}=\ell^{m+m'}{\rm id}_{A_1}.$$
Further, $f\circ \psi=(f\circ \widetilde{f})\circ h=\ell^m h$. Therefore, we may put 
$\psi=\ell^{-m}\psi_1$ as desired. 
\end{proof}

\subsection{Class number of the principal genus for quaternion Hermitian lattices}\label{CN}
In this subsection we refer Section 3.2 of \cite{Ibu20} for the facts and the notation.
Let $p$ be a prime number and $n$ be a positive integer. 
Let $B$ be the definite quaternion algebra ramified only at $p$ and $\infty$ 
(see Proposition 5.1, p.368 of \cite{Pizer-Algo} for an explicit realization). 
We write $B=\Big(\ds\frac{a,b}{\Q}\Big)=\Q\cdot 1\oplus \Q\cdot i\oplus \Q\cdot j \oplus
\Q\cdot ij$ with $i^2=a,\ j^2=b,$ and $ij=-ji$. For each $x=x_0+x_1 i+x_2 j+x_3 ij$, the conjugation of 
$x$ is defined by $\overline{x}=x_0-x_1 i-x_2 j-x_3 ij$. Then $N(x)=x\overline{x},\ 
{\rm Tr}(x)=x+\overline{x}$ are called the norm and the trace of $x$ respectively. 
Let us fix a maximal order $\O$ of $B$ which is also explicitly given in Proposition 5.2 of 
\cite{Pizer-Algo}. 
The pairing $(\ast,\ast)$ on $B^n$ is defined by the following manner:
$$(x,y)=\sum_{i=1}^nx_i\overline{y_i}\ 
\text{for $x=(x_1,\ldots,x_n),\ y=(y_1,\ldots,y_n)\in B^n$}.$$
A submodule $L\subset B^n$ is said to be 
$\O$-lattice if 
\begin{itemize}
\item $L$ is a $\Z$-lattice in $B^n$, hence $L\otimes_\Z\Q=B^n$;
\item $L$ is a left $\O$-module.
\end{itemize}
For an $\O$-lattice $L$, we define by $N(L)$ the two sided (fractional) ideal of $B$ generated by $(x,y),\ x,y\in L$. 
The ideal $N(L)$ is called the norm of $L$. 

For a commutative ring $R$, we extend the conjugation on $\O\subset B$ to 
$\O\otimes_\Z R$ by $\overline{x\otimes r}:=\overline{x}\otimes r$ for each $x\in \O$ and 
$r\in R$. Further, for each $\g=(\g_{ij})\in M_n(\O\otimes_\Z R)$ 
(the set of $n\times n$ matrices over $\O\otimes_\Z R$) we define $\overline{\g}:=
(\overline{\g}_{ij})$. 
We define the algebraic group $G_n$ over $\Z$ which represents the following functor  
from the category of rings to the category of sets: 
$$\underline{G}_n:(Rings)\lra (Sets),\ R\mapsto\underline{G}_n(R):=
\{\g\in M_n(\O\otimes_\Z R)\ |\ \g\cdot {}^t\overline{\g}=\nu(\g)I_n\ \text{for some $\nu(\g)\in R^\times$} \}$$
where $I_n$ stands for the identity matrix of size $n$. 
The similitude map $\nu:G_n\mapsto GL_1$ is defined by $\g\mapsto \nu(\g)$. 
Put $G^1_n:={\rm Ker}(\nu)$ as an algebraic group. 
The group scheme $G_n$ $($resp.\ $G^1_n)$ over $\Z$ is said to be the generalized unitary 
symplectic group $($unitary 
symplectic group$)$ and it is symbolically denoted by $G_n=GUSp_n$ $($resp.\ $G^1_n=USp_n)$. 
It is easy to see that $G_n(\R)$ is compact modulo center and $G^1_n(\R)$ is, in fact, 
compact, since $B$ is definite. By definition, $G_n$ $($resp.\ $G^1_n)$ 
is an inner form of $GSp_n$  $($resp.\ $Sp_n)$.  
\begin{remark}\label{Ibu-history}
Historically, automorphic forms on $G^1_n=USp_n$ for $n\ge 2$ were studied by 
Ihara-Ibukiyama $($see \cite{Ibu18}  and suitable references there$)$. 
After Ibukiyama's joint works with Ihara, he and his collaborators have pursued    
an analogue of Jacquet-Langlands correspondence for $GL_2$. 
\end{remark}
Let us keep introducing some notation. 
For two $\O$-lattices $L_1,L_2$ of $B^n$, they are said to be globally equivalent  
(locally equivalent at a rational prime $p$) if 
$L_1=L_2\g$ for some $\g\in G_n(\Q)$ ($L_1\otimes_\Z\Z_p=(L_2\otimes_\Z\Z_p)\g$  
for some $\g\in G_n(\Q_p)$).  
We also say $L_1$ and $L_2$ belong to the same genus if $L_1$ is 
locally equivalent to $L_2$ for each rational prime $p$. 
For each $\O$-lattice $L$, we denote by $\mL(L)$ the set of 
all $\O$-lattices $L_1$ such that $L_1$ and $L$ belong to the same genus. 
The set $\mL(L)$ is called a genus and we denote by $\mL(L)/\sim$ the 
set of globally equivalent classes of $\mL(L)$. 
\begin{Def}\label{pg}For each $\O$-lattice $L$, the cardinality $H(L)$ 
of $\mL(L)/\sim$ is called the class number of $L$. 
In particular, $\mL(\O^n)/\sim$ is said to be the principal genus class and put 
$$H_n(p,1):=H(\O^n)=|\mL(\O^n)/\sim|.$$
\end{Def}
Let $\A_\Q$ be the ring of adeles of $\Q$ and $\A_f$ be the finite part of $\A_\Q$. 
For an $\O$-lattice $L$ and each rational prime $p$, put $K_p(L):=
\{\gamma_p\in G_n(\Q_p)\ |\ (L\otimes_\Z\Z_p)\gamma_p=L\otimes_\Z\Z_p\}$ which is an open compact subgroup of 
$G_n(\Q_p)$. Then $K(L):=\ds\prod_{p}K_p(L)$ makes up an 
open compact subgroup of $G_n(\A_f)$. For each element 
$\gamma=(\gamma_p)_p$ of $G_n(\A_\Q)$ and an $\O$-lattice $L$, put 
$$L\gamma:=\bigcap_{p<\infty}L\gamma_p\cap B^n$$
and it is easy to see that $L\gamma$ is also an $\O$-lattice which is locally equivalent to 
$L$ at each prime $p$. Hence we have  
\begin{equation}\label{desc1}
K(L)\bs G_n(\A_f)/G_n(\Q)\stackrel{\sim}{\lra} \mL(L)/\sim,\ K(L)\gamma G_n(\Q)\mapsto  
[L\gamma]
\end{equation}
where $G(\Q)$ is diagonally embedded in $G_n(\A_f)$ as $h\mapsto (h)_p$.
As for the computation of the class number of the principal genus, 
the case of $n=1$ is due to Eichler \cite{Eichler} (see also Theorem 1.12, p.346 of 
\cite{Pizer-Algo}) and the case of $n=2$ is handled 
by Hashimoto-Ibukiyama \cite{HI}.
\subsection{Ibukiyama-Katsura-Oort-Serre's result in terms of adelic language}\label{iko}
Let us fix a prime $p$ and put $k=\bF_p$. 
For each positive integer, we denote by $SS_g(p)$ the set of all 
isomorphism classes of principally polarized abelian variety over $k$ of 
dimension $g$.  Henceforth we assume $g\ge 2$. 
According to \cite{IKO} we describe $SS_g(p)$ in terms of 
adelic language. 

Let $E_0$ be a supersingular elliptic curve over $k$. 
It is well-known that $B:={\rm End}(E_0)\otimes_\Z\Q$ is 
the definite quaternion algebra ramified only at $p$ and $\infty$ with 
a maximal order $\O={\rm End}(E_0)$. 
For each prime $q$, put $\O_q:=\O\otimes_\Z\Z_q$. 
 Put $A_0=E^g_0$ and 
define a divisor on $A_0$ by 
$$D:=\ds\sum_{i=0}^{g-1} E^i_0\times\{0_{E_0}\}\times E^{g-i-1}_0$$
where $0_{E_0}$ stands for the origin of $E_0$. 
By using suitable parallel transformations, it is easy to see that $(D^g)=g!$. 
It follows from (\ref{rrt}) that $\mL_0:=\O(D)$ is a principal polarization. 
Let us fix the principally polarized abelian variety $(A_0,\mL_0)$. 
Pick another principally polarized abelian variety $(A,\mL)$. 
As explained before, $A$ is isomorphic to $A_0$ and by pulling back $\mL$ to 
$A_0$, there is one to one correspondence between $SS_g(p)$ and the set $PP_g(A_0)$ of 
isomorphism classes of principal polarizations on $A_0$. Therefore, we have 
\begin{equation}\label{mor1}
SS_g(p)\simeq PP_g(A_0)\stackrel{\subset}{\hookrightarrow} {\rm NS}(A_0)  
\stackrel{j}{\hookrightarrow} {\rm End}(A_0)=M_g(\O),
\end{equation}
where $j$ is defined by $j(\mL)=\phi^{-1}_{\mL_0}\circ \phi_{\mL}$ for each class 
$[\mL]\in PP_g(A_0)$. By Proposition 2.8 of \cite{IKO}, 
the image of $SS_g(p)$ under the map (\ref{mor1}) is given by 
\begin{equation}\label{set1}
\{X\in GL_g(\O)\ |\ X={}^t\overline{X}>0\}.
\end{equation}
We remark that 
the Hauptnorm ${\rm HNm}$ in p.144 of \cite{IKO} is nothing 
but the reduced norm of $M_g(\O)$ and for $X\in M_g(\O)$, ${\rm HNm}(X)=1$ 
if and only if $X\in GL_g(\O)$. 
Pick $X$ from the set (\ref{set1}). By Lemma 2.4 of \cite{IKO}, for each prime $q$, 
there exists $\delta_q\in GL_g(\O_q)$ such that $X=\delta_q{}^t\overline{\delta}_q$. 
Consider an $\O$-lattice $L:=\ds\bigcap_q \O^g\delta_q\cap B^n$. 
By Corollary 2.2 of \cite{IKO}, there exists $\g\in GL_g(B)$ such that 
$L=\O^g\g$. Since $\O^g_q\g=L_q=\O^g_q \delta_q=\O^g_q$, 
$h:=\g{}^t \overline{\g}\in GL_g(\O)$ and clearly $h={}^t\overline{h}>0$. 
Therefore, by Lemma 2.3 of \cite{IKO}, we conclude that $[L]\in \mL(\O^g)/\sim$. 
It follows from Lemma 2.3 of \cite{IKO} again that the association from $X$ to $[L]$ is bijection.  
Summing up, we have the following: 
\begin{prop}\label{ikos}$($Ibukiyama-Katsura-Oort-Serre's theorem$)$ There is a one-to-one correspondence between 
$SS_g(p)$ and $ \mL(\O^g)/\sim$. 
\end{prop}  
We denote by $Z_{G_g}\simeq GL_1$ the center of $G_g=GUSp_g$. 
Recall the open compact subgroup $K(\O^g)=\ds\prod_pK_p(\O^g)$ from (\ref{desc1}) 
for $L=\O^g$. For each prime $\ell\neq p$, put $K(\O^g)^{(\ell)}=\ds\prod_{p\neq \ell}K_p(\O^g)$. Clearly, $K(\O^g)=K(\O^g)^{(\ell)}\times G_g(\Z_\ell)$. 
We identify $B_\ell=B\otimes_{\Q}\Q_\ell$ (resp.\ $\O_\ell=\O\otimes_{\Z}\Z_\ell$) with $M_2(\Q_\ell)$ (resp.\ $M_2(\Z_\ell)$). 
Under this identification, we have $G_g(R)=GSp_g(R)$ for $R=\Z_\ell$ or $\Q_\ell$ 
(cf.\ Lemma 4 of \cite{Ghitza}). 
Therefore, for any subring $M$ of $\Q_\ell$, $G_g(M)$ is naturally identified with 
a subgroup of $G_g(\Q_\ell)=GSp_g(\Q_\ell)$ under the inclusion $M\subset \Q_\ell$.  
\begin{prop}\label{ikos-re} For each prime $\ell\neq p$, 
there is a one-to-one correspondence between 
$SS_g(p)$ and $$G_g(\Z[1/\ell])\bs GSp_g(\Q_\ell)/Z_{GSp_g}(\Q_\ell) GSp_g(\Z_\ell).$$ 
\end{prop}
\begin{proof}For any algebraic closed field $F$, $G^1_g(F)=USp_g(F)=Sp_g(F)$. 
Since $Sp_g$ is simply connected as a group scheme over $\Z$, so is $G^1_g=USp_g$. 
Let $\A^{(\ell)}_f$ be the finite adeles of $\Q$ outside $\ell$. 
By the strong approximation theorem (cf.\ Theorem 7.12, p.427 in Section 7.4 of \cite{PR}) for $G^1_g$ with respect to $S=\{\infty,\ell\}$ and using the exact sequence 
$$1\lra G^1_g\lra G_g\stackrel{\nu}{\lra} GL_1\lra 1,$$
we have a decomposition 
\begin{equation}\label{sapprox1}
G_g(\A_f)=G_g(\A^{(\ell)}_f)\times G_g(\Q_\ell)=G_g(\Q)(K(\O^g)^{(\ell)}\times 
G_g(\Q_\ell)).
\end{equation}
Combining (\ref{desc1}) with (\ref{sapprox1}), we have 
\begin{eqnarray}\label{sapprox2}
\mL(\O^g)/\sim&\stackrel{\sim}{\lla}&K(\O^g)\bs G_g(\A_f)/G_g(\Q) \nonumber \\
&\simeq & G_g(\Q)\bs G_g(\A_f)/K(\O^g)\\
&=&G_g(\Q)\bs (G_g(\Q)(K(\O^g)^{(\ell)}\times 
G_g(\Q_\ell)))/K(\O^g) \nonumber \\
&=&G_g(\Z[1/\ell])\bs GSp_g(\Q_\ell)/
GSp_g(\Z_\ell) \nonumber \\
&=&G_g(\Z[1/\ell])\bs GSp_g(\Q_\ell)/Z_{GSp_g}(\Q_\ell) 
GSp_g(\Z_\ell). \nonumber 
\end{eqnarray}
We remark that at the last line $Z_{GSp_g}(\Q_\ell)$ is intentionally inserted due to the formulation 
in terms of Bruhat-Tits building handled later on. Further, the centers of 
$G_g(\Z[1/\ell])$ and $GSp_g(\Z_\ell)$ are $\Z[1/\ell]^\times$ and $\Z^\times_\ell$ respectively. The equality $\Z[1/\ell]^\times\Z^\times_\ell=\Q^\times_\ell$  explains how $Z_{GSp_g}(\Q_\ell)$ shows up there. 
We have also used $K(\O^g)^{(\ell)}\cap G_g(\Q)=G_g(\Z[1/\ell])$ to obtain 
the fourth line. 
\end{proof}
\subsection{Another formulation due to Jordan-Zaytman}\label{AFJZ}
Let $\ell\neq p$ be a prime. 
Both of $SS_g(p)$ and the Bruhat-Tits building $GSp_g(\Q_\ell)/Z_{GSp_g}(\Q_\ell) 
GSp_g(\Z_\ell)$ endowed with Hecke theory at $\ell$. However, it is not transparent 
to see the compatibility of Hecke actions on both sides under the one-to-one correspondence (\ref{sapprox2}). To overcome this, due to Jordan-Zaytman \cite{JZ}, we use another formulation of $SS_g(p)$ and its connection to $SS_g(p,\ell,A_0,\mL_0)$ by using 
$\ell$-adic Tate modules. 

Pick $(A,\mL)$ from a class in $SS_g(p)$. For a positive integer $n$, 
let 
\[
A[\ell^n]:=\{P\in A(\bF_p)\ |\ \ell^n P=0_A\}\simeq (\Z/\ell^n \Z)^{\oplus 2g}
\] 
and 
put $A[\ell^\infty]=\ds\bigcup_{n\ge 1}A[\ell^n]$. 
We denote by $T_\ell(A)$ the $\ell$-adic Tate module and 
 by $V_\ell(A):=T_\ell(A)\otimes_{\Z_\ell}\Q_\ell$ the $\ell$-adic rational Tate module 
 (cf. Section 18 of Chapter IV of \cite{Mum}).  
 Let us define the coefficient ring $R_V$ to be $\Z/\ell^n\Z$ if $V=A[\ell^n]$, $\Z_\ell$ if $V=T_\ell(A)$, and 
 $\Q_\ell$ if $V=V_\ell(A)$. 
The principal polarization $\phi_\mL:A\stackrel{\sim}{\lra}\wA$ yields 
$V\simeq V^\ast={\rm Hom}_{R_V}(V,R_V)$
and it induces a non-degenerate alternating pairing 
$$\langle \ast,\ast \rangle:V\times V\lra R_V.$$
Let $C$ be a maximal isotropic subgroup of $A[\ell^n]$ for some $n\ge 1$. 
Consider the exact sequence 
\[
0\lra T_\ell(A)\stackrel{\subset}{\lra} V_\ell(A)\stackrel{\pi}{\lra}V_\ell(A)/T_\ell(A)\simeq A[\ell^\infty]
\lra 0.
\]
Then, $T_C:=\pi^{-1}(C)$ is a lattice of $V_\ell(A)$. The quotient $A_C:=A/C$ is also 
a superspecial abelain variety and the line bundle $\mL$ is uniquely descend 
to a principal polarization $\mL_C$ on $A_C$ by 
Corollary of Theorem 2 in Section 23 of Chapter IV of \cite{Mum} 
(see also Proposition 11.25 of \cite{vdGM} for the uniqueness). 
It follows from this that $T_C\simeq T_\ell(A_C)$ has a symplectic $\Z_\ell$-basis $\{f_{C,i}\}_{i=1}^{2g}\subset \Q^{2g}_\ell$ which means the matrix $P_C:=(f_{C,1},\ldots,f_{C,2g})\in M_{2g}(\Q_\ell)$ belongs to $GSp_g(\Q_\ell)$. 
Another choice of a symplectic $\Z_\ell$-basis of $T_C$ yields $P_C\g$ for 
some $\g\in GSp_g(\Z_\ell)$. 
For each $h\in {\rm End}(A)\otimes_\Z \Z[1/\ell]$ which is invertible 
(hence $h$ is an isogeny of degree a power of $\ell$), 
we see easily that $P_{h(C)}=h^\ast P_C$ 
where $h^\ast$ is the endomorphism of $V_\ell(A)$ induced from $h$. 
In fact, it follows from the functorial property of the pairing (see p.228 of \cite{Mum}). 
We identify $G_g(\Z[1/\ell])$ with 
\begin{equation}\label{dagger}
\G(A)^\dagger:=\{f\in ({\rm End}(A)\otimes_\Z \Z[1/\ell])^\times\ |\ f\circ f^\dagger=
f^\dagger\circ f\in \Z[1/\ell]^\times {\rm id}_A\}
\end{equation}
under the natural inclusion $({\rm End}(A)\otimes_\Z \Z[1/\ell])^\times\hookrightarrow 
{\rm Aut}((V_\ell(A),\langle\ast,\ast\rangle))=GSp_g(\Q_\ell)$.

Fix $(A,\mL)$ in a class of $SS_g(p)$. 
We introduce the following sets which play an important role in the 
construction of the isogeny graphs: 
\begin{equation}\label{iso-ell}
{\rm Iso}_{\ell^\infty}(A,\mL):=\{ [(A_C,\mL_C)]\in SS_g(p)  \ |\ n\ge 1,\ C\subset A[\ell^n]:\text{a maximal isotropic subgroup}\}.
\end{equation}
and 
\begin{equation}\label{iso-mark}
SS_g(p,\ell,A,\mL):=\{ [(B,\mM,\phi_B)] \ |\ [(B,\mM)]\in SS_g(p)\}
\end{equation}
where $\phi_B:A\lra B$ is an $\ell$-marking and 
$ [(B,\mM,\phi_B)] $ stands for the equivalent class of $(B,\mM,\phi_B)$. 
Here such two objects  $(A_1,\mL_1,\phi_{A_1})$ and $(A_2,\mL_2,\phi_{A_2})$ are
said to be equivalent if there exists an isomorphism $f:(A_1,\mL_1)\lra (A_2,\mL_2)$ such that 
$f\circ \phi_{A_1}$ and $\phi_{A_2}$ differ by only an element in $\G(A_1)^\dagger$. 
By definition, the natural map from $SS_g(p,\ell,A,\mL)$ to  ${\rm Iso}_{\ell^\infty}(A,\mL)$ 
is surjective while ${\rm Iso}_{\ell^\infty}(A)$ is included in 
$SS_g(p)$.
With the above observation, we have obtained a map 
\begin{equation}\label{another-desc}
{\rm Iso}_{\ell^\infty}(A,\mL)\lra 
G_g(\Z[1/\ell])\bs GSp_g(\Q_\ell)/GSp_g(\Z_\ell),\ 
 [(A_C,\mL_C)] \mapsto 
G_g(\Z[1/\ell])P_C GSp_g(\Z_\ell)
\end{equation}
We then show a slightly modified version of Jordan-Zaytman's theorem, Theorem 46 of \cite{JZ} 
in conjunction with $SS_g(p,\ell,A,\mL)$. 
\begin{thm}\label{JZ} 
Fix $(A,\mL)$ in a class of $SS_g(p)$. 
Keep the notation being as above.  It holds that 
${\rm Iso}_{\ell^\infty}(A,\mL)=SS_g(p)$ and the map $($\ref{another-desc}$)$ induces
 a bijection 
$${\rm Iso}_{\ell^\infty}(A,\mL) 
\stackrel{\sim}{\lra} G_g(\Z[1/\ell])\bs GSp_g(\Q_\ell)/GSp_g(\Z_\ell)=G_g(\Z[1/\ell])\bs GSp_g(\Q_\ell)/Z_{GSp_g}(\Q_\ell) GSp_g(\Z_\ell).$$ 
Further, the natural map $SS_g(p,\ell,A,\mL)\lra {\rm Iso}_{\ell^\infty}(A,\mL)$ is also bijective. 
\end{thm}
\begin{proof}
Surjectivity of (\ref{another-desc}) follows in reverse from the construction by using 
Corollary of Theorem 2 in Section 23 of Chapter IV of \cite{Mum} to guarantee 
the existence of a principal polarization. 
By Proposition \ref{ikos-re} and ${\rm Iso}_{\ell^\infty}(A,\mL)\subset SS_g(p)$, we have 
$$|SS_g(p)|=|G_g(\Z[1/\ell])\bs GSp_g(\Q_\ell)/Z_{GSp_g}(\Q_\ell) GSp_g(\Z_\ell)|\le 
|{\rm Iso}_{\ell^\infty}(A,\mL)| \le |SS_g(p)|$$
and it yields first two claims. With a natural surjection $SS_g(p,\ell,A, \mL)\lra
 {\rm Iso}_{\ell^\infty}(A,\mL)$ and (\ref{another-desc}), we have a surjective map 
$$SS_g(p,\ell,A,\mL)
\lra G_g(\Z[1/\ell])\bs GSp_g(\Q_\ell)/Z_{GSp_g}(\Q_\ell) GSp_g(\Z_\ell).$$
However, by construction and the identification  $({\rm End}(A)\otimes_\Z \Z[1/\ell])^\times=
G_g(\Z[1/\ell])$, two objects of $SS_g(p,\ell,A,\mL)$ which go to one element in 
the target differ by only $\ell$-markings. Therefore, the above map is bijective. 
Hence, $SS_g(p,\ell,A,\mL)\stackrel{\sim}{\lra}{\rm Iso}_{\ell^\infty}(A,\mL) =SS_g(p)$. 

Note that the factor $Z_{GSp_g}(\Q_\ell)\simeq \Q^\times_\ell $ is intentionally 
inserted in front of $GSp_g(\Z_\ell)$ as explained in the proof of Proposition \ref{ikos-re}. 
\end{proof}

As a byproduct we have 
\begin{cor}\label{connected-ness}Let $\ell$ be a prime different from $p$. 
Let $\Gc^{SS}_{g}(\ell,p)$ is the isogeny graph defined in Section \ref{intro}. 
Then, $\Gc^{SS}_{g}(\ell,p)$ is a connected graph. 
\end{cor}
\begin{proof}
By the proof of Theorem \ref{JZ}, we have 
$SS_g(p,\ell,A,\mL)\stackrel{\sim}{\lra}{\rm Iso}_{\ell^\infty}(A,\mL) =SS_g(p)$ for any fixed $(A,\mL)$ in a class of $SS_g(p)$. 
This means that any two classes are connected by isogenies of degree a power of 
$\ell$ and such an isogeny can be written as a composition of some $(\ell)^g$-isogenies 
by Theorem 34 of \cite{JZ}. 
This shows the claim. 
\end{proof}

\subsection{The Hecke operator at $\ell$}\label{hecke-at-ell}
Finally we discuss a relation of the map (\ref{another-desc}) with the 
Hecke operator at $\ell$. 
We refer Section 3 in Chapter VII of \cite{CF} for general facts and Section 16 through 19 of 
\cite{vdGeer} as a reader's friendly reference. 
For each prime $\ell$ different from $p$ and a class $[(A,\mL,\phi_A)]\in 
SS_g(p,\ell,A_0,\mL_0)$, we define 
the (geometric) Hecke correspondences $T(\ell)^{{\rm geo}}_{(A_0,\mL_0)}$ at $\ell$: 
\begin{equation}\label{geo-Hecke}
T(\ell)^{{\rm geo}}_{(A_0,\mL_0)}([(A,\mL,\phi_A)]):=\sum_{C\subset A[\ell]\atop 
\text{maximal isotropic}}[(A_C,\mL_C,f_C\circ \phi_A)]. 
\end{equation}
where $f_C:A\lra A_C$ is the natural projection. Similarly, we also define 
the (geometric) Hecke correspondences $T(\ell)^{{\rm geo}}$ at $\ell$ on $SS_g(p)$: 
\begin{equation}\label{geo-Hecke1}
T(\ell)^{{\rm geo}}([(A,\mL)]):=\sum_{C\subset A[\ell]\atop 
\text{maximal isotropic}}[(A_C,\mL_C)]. 
\end{equation}

Recall $GSp_g(\Q_\ell)=GSp(\Q^{2g}_\ell,\ \langle \ast,\ast \rangle)$ where 
$\langle \ast,\ast \rangle$ is the standard symplectic pairing on $\Q^{2g}_\ell\times
\Q^{2g}_\ell$. Put $V=\Q^{2g}_\ell$. 
As seen before, each element of $GSp_g(\Q_\ell)/GSp_g(\Z_\ell)$ can be regarded as a lattice 
$L$ of $V$ such that $\langle \ast,\ast \rangle_{L\times L}$ gives a $\Z_\ell$-integral  
symplectic structure on $L$. 
Using this interpretation, 
 each element of $GSp_g(\Q_\ell)/Z_{GSp_g}(\Q_\ell) GSp_g(\Z_\ell)$ can be regard as 
 a homothety class $[L]$ for such an $L$. 
For each $L$ being as above, we define the Hecke correspondence on 
$GSp_g(\Q_\ell)/GSp_g(\Z_\ell)$ at $\ell$
\begin{equation}\label{Hecke-ope}
T(\ell)([L]):=\sum_{L\subset L_1 \subset \ell^{-1}L  \atop  
L_1/L\text{:maximal isotropic}}[L_1]
\end{equation}
where $L_1$ runs over all lattice enjoying 
$L\subset L_1 \subset \ell^{-1}L$ as denoted and that 
$L_1/L$ is a maximal isotropic subgroup of $\ell^{-1}L/L$ 
with respect to the symplectic pairing $\langle \ast,\ast \rangle_{\ell^{-1}L/L\times\ell^{-1}L/L}$. 
Clearly, the action of $G_g(\Z[1/\ell])$ (given by multiplication from the left) on lattices are equivariant under 
$T(\ell)$. Therefore, it also induces a correspondence on 
$G_g(\Z[1/\ell])\bs GSp_g(\Q_\ell)/Z_{GSp_g}(\Q_\ell) GSp_g(\Z_\ell)$ 
and by abusing notation, we denote it by $T(\ell)$. 
For a set $X$, we write ${\rm Div}(X)_\Z:=\bigoplus_{P\in X}\Z P$. 
The identification (\ref{another-desc}) with the bijection 
\begin{equation}\label{forget}
SS_g(p,\ell,A_0,\mL_0)\stackrel{\sim}{\lra}SS_g(p),\ [(A,\mL,\phi_A)]\mapsto [(A,\mL)]
\end{equation} 
yields a bijection
\begin{equation}\label{another-desc1}
SS_g(p,\ell,A_0,\mL_0)\lra 
G_g(\Z[1/\ell])\bs GSp_g(\Q_\ell)/GSp_g(\Z_\ell). 
\end{equation}
Then we have obtained the following:
\begin{thm}\label{Rel-HP}The following diagram is commutative:

\[
  \begin{CD}
   {\rm Div}(SS_g(p))_\Z @<{(\ref{forget})\atop \sim}<<   {\rm Div}(SS_g(p,\ell,A_0,\mL_0))_\Z 
   @>{(\ref{another-desc1})\atop \sim}>>  
   {\rm Div}( G_g(\Z[1/\ell])\bs GSp_g(\Q_\ell)/Z_{GSp_g}(\Q_\ell) GSp_g(\Z_\ell))_\Z\\
  @V{T(\ell)^{{\rm geo}}}VV  @V{T(\ell)^{{\rm geo}}_{(A_0,\mL_0)}}VV    @V{T(\ell)}VV \\
   {\rm Div}(SS_g(p))_\Z @<{(\ref{forget})\atop \sim}<<    {\rm Div}(SS_g(p,\ell,A_0,\mL_0))_\Z   @>{(\ref{another-desc1})\atop \sim}>>   {\rm Div}(G_g(\Z[1/\ell])\bs GSp_g(\Q_\ell)/Z_{GSp_g}(\Q_\ell) 
      GSp_g(\Z_\ell))_\Z.
  \end{CD}
\]
\end{thm}

\subsection{The Hecke action and automorphisms}
In this subsection we describe the behavior of the Hecke action of $T(\ell)$ on 
the finite set $$G_g(\Z[1/\ell])\bs GSp_g(\Q_\ell)/GSp_g(\Z_\ell)=
G_g(\Z[1/\ell])\bs GSp_g(\Q_\ell)/Z_{GSp_g}(\Q_\ell)GSp_g(\Z_\ell)$$ in terms of 
automorphism groups of objects in $SS_g(p,\ell,A_0,\mL_0)$. 

Put $\Gamma=G_g(\Z[1/\ell]),\ G=GSp_g(\Q_\ell)$, $Z=Z_{GSp_g}(\Q_\ell)$ and $K=GSp_g(\Z_\ell)$ for simplicity.
We write 
$$\G\bs G/K=\{\G x_1ZK,\ldots, \G x_h ZK \},\ x_1,\ldots x_h\in G$$
where $h=h_g(p,1)=|\G\bs G/ZK|$. 
For each $i\in \{1,\ldots,h\}$, the coset 
$\G x_i ZK$ is naturally identified with 
$$\G/\G\cap x_i ZK x^{-1}_i=(\G Z/Z)/((\G\cap x_i ZK x^{-1}_i)Z/Z).$$
\begin{lem}\label{auto}Keep the notation being as above. 
Let $(A_i,\mL_i,\phi_{A_i})$ be an element in the class 
corresponding to $\G x_i K$.
There is a natural group isomorphism between 
$\widetilde{\G}_i:=(\G\cap x_i ZK x^{-1}_i)Z/Z$ and ${\rm Aut}((A_i,\mL_i))/\{\pm 1\}$ 
where ${\rm Aut}((A_i,\mL_i))$ is the group of automorphisms of $(A_i,\mL_i)$. 
\end{lem}
\begin{proof}By construction, we have $T_\ell(A_i)=x_i\Z^{2g}_\ell$ under the inclusion $T_\ell(A_i)\hookrightarrow 
V_\ell(A_0)=\Q^{2g}_\ell$ induced by the $ell$-marking of $(A_i,\mL_i)$. 
Then the group $(\G\cap x_i ZK x^{-1}_i)$ obviously acts on $T_\ell(A_i)$. Thus, we have 
an injection $(\G\cap x_i ZK x^{-1}_i)\subset {\rm End}(T_\ell(A_i))$. On the other hand, by Faltings' theorem 
(cf. Theorem 4 of \cite{Fal}),  ${\rm End}(T_\ell(A_i))\simeq {\rm End}(A_i)\otimes_\Z\Z_\ell$. 
Hence we may have $(\G\cap x_i ZK x^{-1}_i)\subset  {\rm End}(A_i)\otimes_\Z\Z_\ell$ which 
is compatible with the identification $\G\subset \G^\dagger(A_i)$. 
Since each element of $\G^\dagger(A_i)$ is an $\ell$-isogeny, it preserves the polarization of $A_i$ up 
to the multiplication by $Z$. It follows from this that  $\widetilde{\G}_i\subset 
{\rm Aut}((A_i,\mL_i))/\{\pm 1\}$. The opposite inclusion follows by Faltings' theorem again. 
\end{proof}
Next we study the image of each element of $\G\bs G/K=\G\bs G/ZK$ under the Hecke action of $T(\ell)$. 
Since $T(\ell)$ is defined in terms of lattices (see (\ref{Hecke-ope})), we define another formulation in terms of elements in $G$.  
Let $t_\ell:={\rm diag}(\overbrace{1,\ldots,1}^{g},\overbrace{\ell,\ldots,\ell}^{g})\in G$.   
We decompose 
\begin{equation}\label{double-coset}
Kt_\ell K=\coprod_{t\in T}g_t K
\end{equation}
where $T$ is the index set so that $|T|=N_g(\ell)$. 
For each $i,j\in \{1,\ldots,h\}$ we define 
\begin{equation}\label{mult}
m_{ij}:=\{t\in T\ |\ \G x_i g_t ZK=\G x_j ZK\}
\end{equation}
which is independent of the choice of the representatives $\{g_t\}_{t\in T}$. 
Let $W(\ell):=\{g_t ZK\ |\ t\in T\}$. 
Then for each $i\in \{1,\ldots,h\}$, recall $\widetilde{\G}_i=(\G\cap x_i ZK x^{-1}_i)Z/Z$. and 
the finite group $x_i^{-1}\widetilde{\G}_i x_i\subset KZ/Z$ acts on $W(\ell)$ 
from the left by multiplication. 
The action induces the orbit decomposition 
\begin{equation}\label{decom-w}
W(\ell)=\coprod_{t\in T'}O_{x_i^{-1}\widetilde{\G}_i x_i}(g_tKZ)
\end{equation}
for some subset $T'\subset T$. 
\begin{lem}\label{stab-auto1}
Keep the notation being as above. 
For each $i\in \{1,\ldots,h\}$ and $t\in T'$, if 
$ \G x_i g_t ZK=\G x_j ZK$ for some $j\in \{1,\ldots,h\}$, 
the stabilizer ${\rm Stab}_{x_i^{-1}\widetilde{\G}_i x_i}(g_tKZ)$ 
is isomorphic to a subgroup $S_i$  of  $\widetilde{\G}_j$.   
\end{lem}
\begin{proof}
By assumption, $x_j=\gamma x_i g_t z k$ for some $\gamma\in \G,\ z\in Z$, and $k\in K$. 
For each $\alpha Z\in x_i^{-1}\widetilde{\G}_i x_i=(x^{-1}_i\G x_i\cap K)Z/Z$, let us consider the 
element $kg^{-1}_t \alpha g_t k^{-1}Z$ in $G/Z$. By using $x_j=\gamma x_i g_t z k$, we see that the element 
 belongs to $x^{-1}_j \G x_j Z/Z$. Further, if $\alpha Z$ is an element of ${\rm Stab}_{x_i^{-1}\widetilde{\G}_i x_i}(g_t KZ)$, 
$kg^{-1}_t \alpha g_t k^{-1}Z$ also belongs to $K$. Therefore, we have a group homomorphism  
$${\rm Stab}_{x_i^{-1}\widetilde{\G}_i x_i}(g_t KZ)\stackrel{\tiny{\text{the conjugation by $kg^{-1}_t$}}}{\longrightarrow} (x^{-1}_j\G x_j\cap K)Z/Z\simeq \widetilde{\G}_j.$$
Clearly, this map is injective and we  have the claim. 
\end{proof}
We also study the converse of the correspondence from $\G x_i g_t ZK$ to $\G x_i ZK$ 
for each $i\in \{1,\ldots,h\}$. Clearly, $g^{-1}_tZK\in W(\ell)$. 
\begin{lem}\label{stab-auto2}
For each $i\in \{1,\ldots,h\}$ and $t\in T'$, if 
$ \G x_i g_t ZK=\G x_j ZK$ for some $j\in \{1,\ldots,h\}$, then 
$|{\rm Stab}_{x_i^{-1}\widetilde{\G}_i x_i}(g_t KZ)|=
|{\rm Stab}_{x_j^{-1}\widetilde{\G}_j x_j}(g^{-1}_t KZ)|$. 
In particular, it holds 
$$|\widetilde{\G}_j |\cdot |O_{x_i^{-1}\widetilde{\G}_i x_i}(g_tKZ)|=
|\widetilde{\G}_i |\cdot |O_{x_j^{-1}\widetilde{\G}_j x_j}(g^{-1}_tKZ)|.$$
\end{lem}
\begin{proof}
As in the proof of the previous lemma, if we write  $x_j=\gamma x_i g_t z k$, then 
the conjugation by $g_t k^{-1}$ yields the isomorphism from 
${\rm Stab}_{x_j^{-1}\widetilde{\G}_j x_j}(g^{-1}_t KZ)$ to ${\rm Stab}_{x_i^{-1}\widetilde{\G}_i x_i}(g_t KZ)$. 
The claim follows from this. 
\end{proof}
Finally, we study the corresponding results 
in $SS_g(p,\ell,A_0,\mL_0)$ under the identification 
\begin{equation}\label{important-identity}
SS_g(p,\ell,A_0,\mL_0)\lra 
G_g(\Z[1/\ell])\bs GSp_g(\Q_\ell)/GSp_g(\Z_\ell)
\end{equation}
given by Theorem \ref{JZ}. 
We write 
$$SS_g(p,\ell,A_0,\mL_0)=\{w_i=[(A_i,\mL_i,\phi_{A_i})]\ |\ i=1,\ldots,h\}.$$
Let us fix $i\in \{1,\ldots,h\}$ and we denote by 
${\rm LG}_i(\ell)=\{C_t\}_{t\in T}$ the set of all totally maximal isotropic subspace of $A_i[\ell]$ with respect to 
the Weil pairing associated to $\mL_i$. 
Here we use the same index $T$ as $W(\ell)$ defined before. 
  Then the group ${\rm RA}_i:={\rm Aut}((A_i,\mL_i))/\{\pm 1\}$ 
acts on ${\rm LG}(\ell)$ since each element there preserves the polarization. As in (\ref{decom-w}) we also have the decomposition 
$${\rm LG}(\ell)=\coprod_{t\in T'}O_{{\rm RA}_i}(C_t).$$
Suppose $\G x_i ZK$ corresponds to $w_i=[(A_i,\mL_i,\phi_{A_i})]$ under 
(\ref{important-identity}). 
\begin{prop}\label{corr-results}Keep the notation being as above. 
The followings holds.  
\begin{enumerate}
\item The pullback of $\phi_{A_i}$ induces an identification between ${\rm LG}_i(\ell)$ and $W(\ell)$.   
\item Suppose $C_t\in {\rm LG}_i(\ell)$ corresponds to $g_tZK\in W(\ell)$ for $t\in T$ under the above 
identification. Let $f_{C_t}:(A_i,\mL_{A_i})\lra (A_{i,C_t},\mL_{A_{i,C_t}})$ be the $(\ell)^g$-isogeny defined by $C_t$ and suppose 
$[(A_{i,C_t},\mL_{(A_{i,C_t}},f_{C_t}\circ \phi_{A_i})]=w_j$ for some $j\in \{1,\ldots,h\}$ and thus 
$f_{C_t}$ is regarded as an $(\ell)^g$-isogeny from $(A_i,\mL_{A_i})$ to $(A_i,\mL_{A_j})$. 
Let $\widetilde{f}_{C_t}:(A_i,\mL_{A_j})\lra (A_i,\mL_{A_i})$ the $(\ell)^g$-isogeny  obtained in Proposition \ref{dual-marking} 
for $f_{C_t}$. Then it holds 
\begin{itemize}
\item the kernel of $\widetilde{f}$ corresponds to $g^{-1}_tZK$ under the above 
identification,  
\item $|{\rm RA}_i|=|\widetilde{\G}_i|$,
\item $|O_{{\rm RA}_i}(C_t)|= |O_{x_i^{-1}\widetilde{\G}_i x_i}(g_tKZ)|$, 
$|O_{{\rm RA}_j}({\rm Ker}\widetilde{f}_{C_t})|= |O_{x_j^{-1}\widetilde{\G}_j x_j}(g^{-1}_tKZ)|$, and 
\item  $|{\rm RA}_j|\cdot |O_{{\rm RA}_i}(C_t)|=
|{\rm RA}_i|\cdot |O_{{\rm RA}_j}({\rm Ker}\widetilde{f}_{C_t})|$. 
\end{itemize}
\end{enumerate}
\end{prop}
\begin{proof}The claim follows from the construction of (\ref{important-identity}) with Lemma \ref{auto} 
through Lemma \ref{stab-auto2}. 
\end{proof}
We remark that the fourth claim of (2) in the above proposition was proved in Lemma 3.2 of \cite{FS2}.

\section{A comparison between two graphs}\label{Comparison}
In this section we check, by passing to $SS_g(p,\ell,A_0,\mL_0)$, that the graph defined by the special 1-complex
$G_g(\Z[1/\ell])\bs GSp_g(\Q_\ell)/Z_{GSp_g}(\Q_\ell)GSp_g(\Z_\ell)$ is naturally identified with 
Jordan-Zaytman's big isogeny graph in \cite{JZ}. 
\subsection{Jordan-Zaytman's big isogeny graph}
We basically follow the notation in \S 7.1 and \S5.3 of \cite{JZ}. 
The $(\ell)^g$-isogeny (big) graph $Gr_g(\ell,p)$ due to Jordan-Zaytman for $SS_g(p)$  is 
defined as a directed (regular) graph where
\begin{itemize}
\item the set of vertices $V(Gr_g(\ell,p))$ is $SS_g(p)$, and
\item the set of directed edges between two vertices $v_1=[(A_1,\mL_1)]$ and 
$v_2=[(A_2,\mL_2)]$ is the set of equivalence classes of  $(\ell)^g$-isogenies between 
$(A_1,\mL_1)$ and $(A_1,\mL_1)$. 
Here two isogenies $f,h:(A_1,\mL_1)\lra (A_2,\mL_2)$ are said to be equivalent if 
there exist automorphisms $\phi\in {\rm Aut}(A_1,\mL_1)$ and  
$\psi\in {\rm Aut}(A_2,\mL_2)$ such that $\psi\circ h=f\circ \phi$. 
\end{itemize}
The case when $g=1$ is nothing but Pizer's graph $G(1,p;\ell)$ handled in 
\cite{Pizer-graph}.  

\subsection{The ($\ell$-marked) $(\ell)^g$-isogeny graph}
Similarly, the ($\ell$-marked) $(\ell)^g$-isogeny graph $\Gc^{SS}_{g}(\ell,p)$ for 
$SS_g(p,\ell,A_0,\mL_0)$ is defined as a directed (regular) graph where
\begin{itemize}
\item the set of vertices $V(\Gc^{SS}_{g}(\ell,p))$ is $SS_g(p,\ell,A_0,\mL_0)$ and
\item the set of edges between two vertices $v_1$ and $v_2$ is the set of equivalence classes of  $(\ell)^g$-isogenies between corresponding principally polarized superspecial abelian varieties commuting with marking isogenies representing 
$v_1$ and $v_2$ under the identification. In other words, if $v_1$  and $v_2$ correspond to 
$[(A_1,\mL_1,\phi_{A_1})]$ and  $[(A_2,\mL_2,\phi_{A_2})]$ with $\ell$-markings 
$\phi_{A_1}:(A_0,\mL_0)\lra (A_1,\mL_1)$ and 
$\phi_{A_2}:(A_0,\mL_0)\lra (A_2,\mL_2)$ respectively, then 
an edge from $v_1$ to $v_2$ is an $(\ell)^g$-isogeny $f:(A_1,\mL_1)\lra (A_2,\mL_2)$ such that 
two markings $f\circ \phi_{A_1}$ and $\phi_{A_2}$ of $(A_2,\mL_2)$ from $(A_0,\mL_0)$ 
differ by only an element in $\Gamma(A_0)^\dagger$. 
\end{itemize}

\subsection{The graph defined by the special 1-complex}\label{btq1}
Put $\Gamma=G_g(\Z[1/\ell]),\ G=GSp_g(\Q_\ell)$, $Z=Z_{GSp_g}(\Q_\ell)$ and $K=GSp_g(\Z_\ell)$ for simplicity. 
We consider the graph associated to the quotient $\G\bs \mathcal{S}_g$ 
where $\G=G_g(\Z[1/\ell])$ and 
$\mathcal{S}_g=GSp_g(\Q_\ell)/Z_{GSp_g}(\Q_\ell)GSp_g(\Z_\ell)$. 

Two elements $v_1=\Gamma g_1 ZK$ and 
$v_2=\Gamma g_2 ZK$ in $\Gamma\bs G/ZK$ said to be 
adjacent if $v_2=\Gamma g_1 g_t ZK$ for some $t \in T$ where 
$\{g_t\}_{t\in T}$ is defined in (\ref{double-coset}). 

The graph in question, say ${\rm BTQ}^1_g(\ell,p)$, is a directed (regular) graph where
\begin{itemize}
\item the set of vertices $V({\rm BTQ}^1_g(\ell,p))$ is $\Gamma\bs G/ZK$, and
\item the set of directed edges between two vertices $v_1=\Gamma g_1 ZK$ and 
$v_2=\Gamma g_2 ZK$ is defined by the adjacency condition in the above sense. 
Namely, an edge from $v_1$ from $v_2$ is $g_t$ with $t\in T$ such that 
$v_2=\Gamma g_1 g_t ZK$. 
\end{itemize} 

\subsection{Comparison theorem}
Let us keep the notation in this section. 
We define 
$${\rm RA}(v):=
\left\{\begin{array}{ll}
{\rm RA}(A,\mL) & \text{if $v=[(A,\mL)]\in SS_g(p)$ or $v=[(A,\mL,\phi_A)]\in 
SS_g(p,\ell,A_0,\mL_0)$,}\\
(\G\cap x ZK x^{-1})Z/Z & \text{if $v=\G x ZK$ in the case of ${\rm BTQ}^1_g(\ell,p)$.}
\end{array}\right.$$
Further we also define 
$${\rm Ker}(e):=
\left\{\begin{array}{ll}
{\rm Ker}(f) & \text{if $e$ is a class of $(\ell)^g$-isogeny $f$ in the case 
of $SS_g(p)$ or  $SS_g(p,\ell,A_0,\mL_0)$,}\\
g_t & \text{if $e$ is an edge defined by $g_t,\ t\in T$ 
in the case of ${\rm BTQ}^1_g(\ell,p)$.}
\end{array}\right.$$
We will prove the following comparison theorem which plays an important role in 
our study:
\begin{thm}\label{comparison}
The identifications $($\ref{forget}$)$ and $($\ref{another-desc1}$)$ induce the 
following graph isomorphisms
$$Gr_g(\ell,p)\stackrel{(\ref{forget})\atop \sim}{\longleftarrow} \mathcal{G}^{SS}_{g}(\ell,p)
\stackrel{(\ref{another-desc1})\atop\sim}{\lra}{\rm BTQ}^1_g(\ell,p).$$
Further, the following properties are preserved under the isomorphisms:
\begin{itemize}
\item the Hecke action of $T(\ell)^{{\rm geo}},\ T(\ell)^{{\rm geo}}_{(A_0,\mL_0)}$, or 
$T(\ell)$ on each set of the vertices defines $N_g(\ell)$-neighbors of a given 
vertex, 
\item each edge $e$ from $v_1$ to $v_2$ has an opposite $\widehat{e}$ such that  
$$|{\rm RA}(v_2)|\cdot |O_{{\rm RA}(v_1)}({\rm Ker}(e))|=|{\rm RA}(v_1)|\cdot 
|O_{{\rm RA}(v_2)}({\rm Ker}(\widehat{e}))|$$
\end{itemize}
\end{thm}
\begin{proof}
As in the claim already, 
the identifications between the sets of vertices are given by 
(\ref{forget}) and (\ref{another-desc1}). The compatibility of the Hecke operators 
follows from Theorem \ref{Rel-HP} and this yields the first property in the 
latter claim. The remaining formula follows from Proposition \ref{corr-results}. 
\end{proof}
\begin{cor}\label{lwm}Keep the notation being as above. 
The random walk matrices for 
$Gr_g(\ell,p),\ \Gc^{SS}_{g}(\ell,p)$, and  ${\rm BTQ}^1_g(\ell,p)$ 
coincide each other. 
\end{cor}
We remark that Theorem \ref{Rel-HP} is insufficient to prove the above corollary 
while Theorem \ref{comparison} tells us more finer information for the relation 
of the reduced automorphisms and the multiplicity of each edge.  

\begin{remark}As shown in Theorem \ref{comparison} or Section 3 of \cite{FS2}, 
the group of reduced automorphisms gives a finer structure of 
its orbit of a given Lagragian subspace defining an $(\ell)^g$-isogeny. 
The edges in Figure \ref{Fig:p13} can be more precise as in the figure in 7A, p.297 of 
\cite{KT}. 
\end{remark}

\section{Bruhat-Tits buildings for symplectic groups}\label{BT}

In this and the following chapter, we introduce a more general framework than the case to which we apply.
The purpose is to simplify the notations and to indicate that the methods we use are applicable in a wider context.
The reader may assume that $F=\Q_\ell$ and $\varpi=\ell$ in the following discussion.

\subsection{Symplectic groups revisited for the buildings}\label{Sec:symplectic_group}
Let $F$ be a non-archimedean local field of characteristic different from $2$ and $\Oc$ be the ring of integers.
We fix a uniformizer $\varpi$ and identify the residue field $\Oc/\varpi \Oc$ with a finite field $\Fb_q$ of order $q$.
Further we denote by $F^\times$ and $\Oc^\times$ the multiplicative groups in $F$ and $\Oc$ respectively.
Let $\ord_\varpi$ be a discrete valuation in $F$, normalized so that $\ord_\varpi(F^\times)=\Z$.
For example, we consider the $\ell$-adic field $\Q_\ell$ for a prime $\ell$ with the ring of integers $\Z_\ell$,
where $\ell$ is a uniformizer and the residue field is $\Fb_\ell=\Z/\ell \Z$. 

For a positive integer $n$, let $V:=F^{2n}$ be the symplectic space over $F$ equipped with the standard symplectic pairing $\bi-form$ defined by 
$\lbr v, w\rbr=\tr^v J_n w$ for $v, w \in F^{2n}$.
For $V$,
there exists a basis $\{v_1, \dots, v_n, w_1, \dots, w_n\}$ such that
\[
\lbr v_i, w_j\rbr=\d_{i j} \quad \text{and} \quad \lbr v_i, v_j\rbr=\lbr w_i, w_j\rbr=0 \quad \text{for any $i, j=1, \dots, n$},
\]
where $\d_{ij}$ equals $1$ if $i=j$ and $0$ if $i \neq j$,
and we call it a symplectic basis of $(V, \bi-form)$.
Each choice of a symplectic basis yields an isomorphism between the isometry group and $Sp_n(F)$.

Note that the following elements are in $GSp_n(F)$,
\[
t_\lambda:=\diag(1, \dots, 1, \lambda, \dots, \lambda)=
\(
\begin{array}{cc}
I_n & 0\\
0 & \lambda I_n
\end{array}
\) \quad \text{for $\lambda \in F^\times$}.
\]
(See also Section \ref{notation}).
In the subsequent sections, we consider the projectivised groups:
let $PSp_n(F)$ and $PGSp_n(F)$ be the groups $Sp_n(F)$ and $GSp_n(F)$ modulo the centers respectively.
If we naturally identify $PSp_n(F)$ with a normal subgroup of $PGSp_n(F)$, then the quotient group
$PGSp_n(F)/PSp_n(F)$ is isomorphic to $(\Z/2\Z)\times \Oc^\times$, which is generated by the images of $t_\lambda=\diag(1, \dots, 1, \lambda, \dots, \lambda)$ for $\lambda \in \varpi \Oc^\times$.
Similarly, letting $PSp_n(\Oc)$ and $PGSp_n(\Oc)$ be the groups $Sp_n(\Oc)$ and $GSp_n(\Oc)$ modulo the centers respectively,
we identify $PSp_n(\Oc)$ with a subgroup in $PGSp_n(\Oc)$.

\subsection{Bruhat-Tits building: the construction}

Let $(V, \bi-form)$ be a symplectic space over $F$ of dimension $2n$.
We define a lattice $\L$ in $V$ as a free $\Oc$-module of rank $2n$.
Note that if $\L$ is a lattice, then $\L/\varpi \L$ is a vector space over $\Fb_q$ of dimension $2n$.
We say that a lattice $\L$ is primitive
if 
\[
\lbr \L, \L\rbr \subseteq \Oc \quad \text{where } \langle \L, \L\rangle:=\{\langle v, w\rangle \mid v, w \in \L\},
\]
and $\bi-form$ induces a non-degenerate alternating form on $\L/\varpi\L$ over $\Fb_q$.

Let $\L_i$ for $i=1, 2$ be lattices in $V$,
and
we say that they are homothetic
if
\[
\L_1 =\a \L_2 \quad \text{for some $\a \in F^\times$}.
\]
This defines an equivalence relation in the set of lattices in $V$.
We denote the homothety class of a lattice $\L$ by $[\L]$.
Let us define the set $\Lc_n$ of homothety classes $[\L]$ of lattices
such that there exist a representative $\L$ of $[\L]$ and a primitive lattice $\L_0$ 
satisfying that
\[
\varpi \L_0 \subseteq \L \subseteq \L_0 \quad \text{and} \quad \lbr \L, \L\rbr \subseteq \varpi \Oc.
\]
By the definition, if $[\L] \in \Lc_n$, then a representative $\L$ yields a subspace $\L/\varpi \L_0$ of $\L_0/\varpi \L_0$
with some primitive lattice $\L_0$
such that it is \textit{totally isotropic}, i.e., the induced non-degenerate alternating form $\bi-form$ vanishes on $\L/\varpi \L_0$ in $\L_0/\varpi\L_0$.
Further we define the incidence relation in $\Lc_n$ and denote by $[\L_1] \sim [\L_2]$ for two distinct homothety classes
if there exist representatives $\L_i$ of $[\L_i]$ for $i=1, 2$ and a primitive lattice $\L_0$
such that
\[
\varpi \L_0 \subseteq \L_i \subseteq \L_0 \quad \text{for $i=1, 2$},
\]
and either $\L_1 \subseteq \L_2$ or $\L_2 \subseteq \L_1$ holds.

The \textit{Bruhat-Tits building} $\Bc_n$ (in short, {\it building}) for the group $PGSp_n(F)$ (or $Sp_n(F)$)
is the clique complex whose set of vertices $\Ver(\Bc_n)$ is $\Lc_n$,
i.e., $\s \subset \Ver(\Bc_n)$ defines a simplex if any distinct vertices in $\s$ are incident.
The building $\Bc_n$ is a simplicial complex of dimension $n$;
note that each chamber (i.e., a simplex of maximal dimension) $[\L_0], [\L_1], \dots, [\L_n]$ corresponds to 
a sequence of lattices 
\[
\L_0 \subseteq \L_1 \subseteq \cdots \subseteq \L_n \subseteq \varpi^{-1}\L_0,
\]
where $\varpi^{-1}\L_0$ is primitive, 
such that
\[
\{0\} \subseteq \L_1/\L_0 \subseteq \L_2/\L_0 \subseteq \cdots \subseteq \L_n/\L_0 \subseteq \varpi^{-1}\L_0/\L_0
\]
forms a complete flag of a maximal totally isotropic subspace $\L_n/\L_0$ in $\varpi^{-1}\L_0/\L_0$ over $\Fb_q$.

The group $Sp_n(F)$ acts on $\Bc_n$ as simplicial automorphisms:
let us fix a symplectic basis $\{v_1, \dots, v_n, w_1, \dots, w_n\}$ of $(V, \bi-form)$,
which we identify with the standard symplectic space over $F$.
Then the action is defined by $[\L] \mapsto [M \L]$ for $[\L] \in \Ver(\Bc_n)$ and $M \in Sp_n(F)$,
and this action is simplicial since it preserves the incidence relation.
Moreover, this yields the action of the projectivised group $PSp_n(F)$ on $\Bc_n$.

We define the label (or, color) on the set of vertices $\Ver(\Bc_n)$.
For any lattice $\L$, there exists some $\g \in GL_{2n}(F)$
such that $\g u_1, \dots, \g w_n$
form an $\Oc$-basis of $\L$.
Let
\[
\lab_n[\L]:=\ord_\varpi(\det \g) \mod 2n.
\]
Note that this depends only on the homothety class of $\L$
since $\det (\a \g)=\a^{2n}\det(\g)$ for $\a \in F^\times$ and for $\g \in GL_{2n}(F)$, and $\det \g \in \Oc^\times$ for $\g \in GL_{2n}(\Oc)$. 
Hence the function $\lab_n:\Ver(\Bc_n) \to \Z/2n\Z$ is well-defined and we call $\lab_n[\L]$ the \textit{label} of a vertex $[\L] \in \Ver(\Bc_n)$.
For example, let us consider a sequence of lattices $\L_0, \dots, \L_n$, where
\begin{equation}\label{Eq:Lk}
\L_k:=\Oc u_1\oplus \cdots \oplus \Oc u_k \oplus \Oc \varpi u_{k+1}\oplus \cdots \oplus \Oc \varpi w_1\oplus \cdots \oplus \Oc \varpi w_n \quad \text{for $0 \le k <n$},
\end{equation}
and
$\L_n:=\Oc u_1\oplus \cdots \oplus \Oc u_n \oplus \Oc \varpi w_1\oplus \cdots \oplus \Oc \varpi w_n$.
Then
$\L_0 \subseteq \cdots \subseteq \L_n \subseteq \varpi^{-1}\L_0$ and $\varpi^{-1}\L_0$ is primitive,
and since the chain $\L_1/\L_0 \subseteq \cdots \subseteq \L_n/\L_0$ forms a maximal totally isotropic flag in $\varpi^{-1}\L_0/\L_0$ over $\Fb_q$,
the corresponding homothety classes
$[\L_0], \dots, [\L_n]$ define a chamber in $\Bc_n$.
In this case, we have that $\lab_n[\L_k]=2n-k \mod 2n$ for $0 \le k \le n$.
We call the chamber determined by $[\L_0], \dots, [\L_n]$ the \textit{fundamental chamber} $\Cc_0$.
Here we note that $\lab_n$ misses the values $1, \dots, n-1$ in $\Z/2n \Z$.
It is known that $Sp_n(F)$ acts transitively on the set of \textit{chambers} \cite[Section 20.5]{Garrett},
i.e., every chamber is of the form $\g \Cc_0$ for $\g \in Sp_n(F)$.
By definition, the action of $Sp_n(F)$ preserves the labels on $\Ver(\Bc_n)$.
It thus implies that the action is not vertex-transitive for any $n \ge 1$.

\subsection{Apartments}

Let us introduce a system of apartments in the building $\Bc_n$, following \cite[Chapter 20]{Garrett} and \cite{ShemanskeSp}.
A \textit{frame} is an unordered $n$-tuple,
\[
\{\lambda_1^1, \lambda_1^2\}, \dots, \{\lambda_n^1, \lambda_n^2\},
\]
such that each $\{\lambda_i^1, \lambda_i^2\}$ is an unordered pair of lines
which span a $2$-dimensional symplectic subspace with the induced alternating form
for $i=1, \dots, n$, and
\[
V=V_1\oplus \cdots \oplus V_n \quad \text{where $V_i:=\lambda_i^1\oplus\lambda_i^2$ and $V_i \bot V_j$ if $i \neq j$},
\]
i.e., $\langle v, v'\rangle=0$ for all $v \in V_i$ and all $v' \in V_j$ if $i \neq j$.
An \textit{apartment} defined by a frame $\{\lambda_i^1, \lambda_i^2\}$ for $i=1, \dots, n$
is a maximal subcomplex of $\Bc_n$ on the set of vertices $[\L]$ such that
\[
\L=\bigoplus_{i=1}^n \(M_i^1\oplus M_i^2\) \quad \text{where $M_i^j$ is a rank one free $\Oc$-module in $\lambda_i^j$ for $j=1, 2$},
\]
for some (equivalently, every) representative $\L$ in the homothety class.
We define a system of apartments as a maximal set of apartments.

Following \cite{ShemanskeSp},
we fix a symplectic basis $\{u_1, \dots, u_n, w_1, \dots, w_n\}$ of $V$ and a uniformizer $\varpi$ in $F$
and lighten the notation:
we denote a lattice
\[
\L=\Oc \varpi^{a_1}u_1\oplus \cdots \oplus \Oc \varpi^{a_n}u_n \oplus \Oc \varpi^{b_1}w_1\oplus \cdots \oplus \Oc \varpi^{b_n}w_n
\quad \text{for $a_i, b_i \in \Z$, $i=1, \dots, n$},
\]
by $\L=(a_1, \dots, a_n ; b_1, \dots, b_n)$,
and the homothety class by $[\L]=[a_1, \dots, a_n ; b_1, \dots, b_n]$.
For $\L$,
we have $\langle \L, \L\rangle \subset \Oc$ if and only if 
$\langle \varpi^{a_i}u_i, \varpi^{b_i}w_i\rangle=\varpi^{a_i+b_i} \in \Oc$ for all $i=1, \dots, n$.
This is equivalent to that 
$a_i+b_i\ge 0$ for all $i=1, \dots, n$, in which case,
$\L/\varpi\L$ is a non-degenerate alternating space with the induced form over the residue field $\Oc/\varpi\Oc$
if and only if $a_i+b_i=0$ for all $i=1, \dots, n$.

For the fixed basis,
let
$\lambda_i^1:=F u_i$ and $\lambda_i^2:=F w_i$ for $i=1, \dots, n$.
The frame $\{\lambda_i^1, \lambda_i^2\}_{i=1, \dots, n}$ determines an apartment $\SS_0$ in the building $\Bc_n$ for $Sp_n(F)$.
We call $\SS_0$ the \textit{fundamental apartment}.
The chain of lattice $\L_0 \subseteq \cdots \subseteq \L_n$ in \eqref{Eq:Lk} defines a chamber $\Cc_0$ in $\SS_0$ containing $[\L_0]$:
\begin{align*}
\L_0=(1, \dots, 1 ; 1, \dots, 1) &\subset (0, 1, 1, \dots, 1 ; 1, \dots, 1)\\
												&\subset (0, 0, 1, \dots, 1 ; 1, \dots, 1) \subset \cdots \subset (0, 0, \dots, 0 ; 1, \dots, 1) \subset \varpi^{-1}\L_0.
\end{align*}
Moreover, the following chain
\begin{align*}
\L_0=(1, \dots, 1 ; 1, \dots, 1) &\subset (1, 0, 1, \dots, 1 ; 1, \dots, 1)\\
												&\subset (0, 0, 1, \dots, 1 ; 1, \dots, 1) \subset \cdots \subset (0, 0, \dots, 0 ; 1, \dots, 1) \subset \varpi^{-1}\L_0,
\end{align*}
where the lattices are the same as above except for the second one,
defines a chamber which shares a codimension one face with $\Cc_0$.

We shall see the rest of chambers in the apartment $\SS_0$ by an action of the affine Weyl group attached to the building.
Denoting by $\Nc_0$ and by $\Ic_0$ the subgroups preserving $\SS_0$ (as a set) and $\Cc_0$ (pointwise) in $Sp_n(F)$ respectively,
the affine Weyl group is isomorphic to $\Nc_0/(\Nc_0\cap \Ic_0)$, which naturally acts on the chambers in $\SS_0$ transitively.
For $Sp_n(F)$, the affine Weyl group is of type $\wt C_n$ with the Coxeter diagram 
\begin{center}
\scalebox{0.9}[0.9]{
\begin{tikzpicture}
\draw[very thick, double] (0, 0) -- (2, 0);
\draw[very thick] (2, 0) -- (4, 0);
\draw[very thick, dotted] (4, 0) -- (6, 0);
\draw[very thick] (6, 0) -- (8, 0);
\draw[very thick, double] (8, 0) -- (10, 0);
\filldraw[black] (0, 0) circle (2pt)
node[anchor=south]{$1$};
\filldraw[black] (2, 0) circle (2pt)
node[anchor=south]{$2$};
\filldraw[black] (4, 0) circle (2pt)
node[anchor=south]{$3$};
\filldraw[black] (6, 0) circle (2pt)
node[anchor=south]{$n-1$};
\filldraw[black] (8, 0) circle (2pt)
node[anchor=south]{$n$};
\filldraw[black] (10, 0) circle (2pt)
node[anchor=south]{$n+1$};
\end{tikzpicture}
}
\end{center}
on $(n+1)$ vertices.
Each vertex $i$ in the Coxeter diagram corresponds to a reflection $s_i$ satisfying that
$s_i^2=1$ and $s_i s_j$ has order $m_{ij}$,
where
\[
m_{12}=m_{n(n+1)}=4, \quad m_{i(i+1)}=3 \quad \text{for $i \neq 1, n$, and} \quad m_{ij}=2 \quad \text{otherwise}.
\]
The affine Weyl group of type $\wt C_n$ is generated by $s_1, \dots, s_{n+1}$.
Given the symplectic basis $\{u_1, \dots, u_n, w_1, \dots, w_n\}$, the action on it is realized as in the following:
\begin{itemize}
\item[] $s_1$ exchanges $u_n$ and $w_n$, and fixes the others,
\item[] $s_j$ $(2 \le j \le n)$ exchanges $u_{n-j+1}$ and $u_{n-j+2}$, and $w_{n-j+1}$ and $w_{n-j+2}$ simultaneously and fixes the others, and
\item[] $s_{n+1}$ maps $u_1$ to $\varpi w_1$ and $w_1$ to $\varpi^{-1}u_1$ and fixes the others.
\end{itemize}
In the fundamental apartment $\SS_0$, denoting a vertex by $[a_1, \dots, a_n ; b_1, \dots, b_n]$,
we have that
\begin{align*}
&s_1[a_1, \dots, a_n ; b_1, \dots, b_n]=[a_1, \dots, a_{n-1}, b_n ; b_1, \dots, b_{n-1}, a_n],\\
&s_j[a_1, \dots, a_n ; b_1, \dots, b_n]\\
&=[a_1, \dots, a_{n-j+2}, a_{n-j+1}, \dots, a_n ; b_1, \dots, b_{n-j+2}, b_{n-j+1}, \dots, b_n] \quad \text{for $2 \le j \le n$},\\
&\text{and}\\
&s_{n+1}[a_1, \dots, a_n ; b_1, \dots, b_n]=[b_1-1, a_2, \dots, a_n ; a_1+1, b_2, \dots, b_n].
\end{align*}
A direct computation shows that $s_1, \dots, s_{n+1}$ satisfy the indicated Coxeter data.
Deleting either $s_1$ or $s_{n+1}$ yields a group isomorphic to the spherical Weyl group of type $C_n$.
Note that
the vertex $v=[a_1, \dots, a_n ; b_1, \dots, b_n]$ is fixed by $s_i$ for $1 \le i \le n$
if and only if $a_i=b_j$ for all $1 \le i, j \le n$,
and $v$ is fixed by $s_i$ for $2 \le i \le n+1$ if and only if $a_i=b_i-1$ for all $1 \le i \le n$.
In the fundamental chamber $\Cc_0$, such vertices are $[\L_0]=[1, \dots, 1 ; 1, \dots, 1]$ and $[\L_n]=[0, \dots, 0 ; 1, \dots, 1]$ respectively. 
Although we do not use the fact, it is useful to note that the spherical Weyl group $C_n$ is isomorphic to the signed permutation group $(\Z/2 \Z)^n \rtimes \mathfrak{S}_n$ 
whose order is $2^n n!$.

\begin{exa}
If $n=2$, then we have $8$ chambers containing vertex $[\L_0]=[1,1;1,1]$ in a fixed apartment,
where the fundamental chamber $\Cc_0$ is defined by the chain
\[
\L_0=(1,1;1,1) \subset (0,1;1,1) \subset (0,0;1,1) \subset (0,0;0,0)=\varpi^{-1}\L_0.
\]
The locations of chambers $\Cc_0, s_1\Cc_0, s_2 \Cc_0$ and $s_3 \Cc_0$ are indicated for generators $s_1, s_2, s_3$ of $\wt C_2$ in Figure \ref{Fig:apartment}.

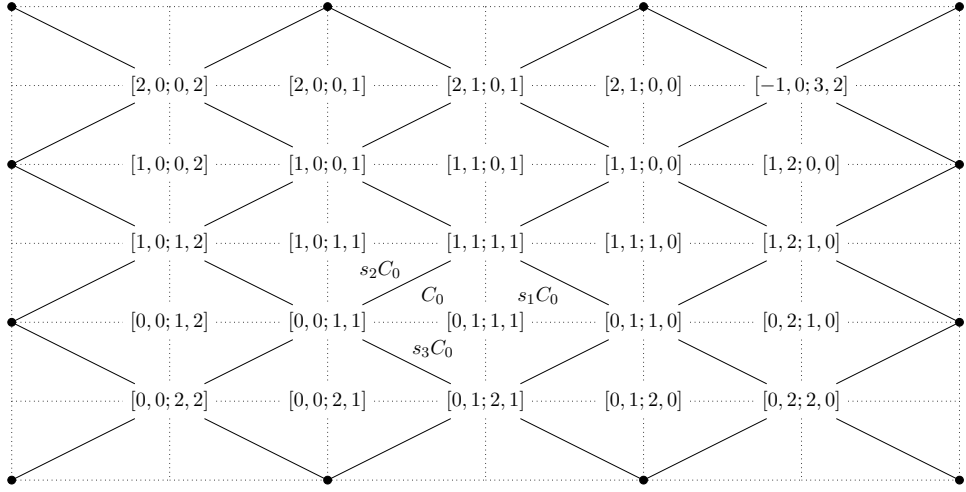
\begin{figure}\label{Fig:apartment}
\centering
\scalebox{0.7}[0.7]{
\begin{tikzpicture}[dot/.style={draw, circle, scale=0.4, fill=black}]

\node[dot] at (-9, 4.5) (-9, 4.5) {};
\node[] at (-6, 4.5) (-6, 4.5) {};
\node[dot] at (-3, 4.5) (-3, 4.5) {};
\node[] at (0, 4.5) (0, 4.5) {};
\node[dot] at (3, 4.5) (3, 4.5) {};
\node[] at (6, 4.5) (6, 4.5) {};
\node[dot] at (9, 4.5) (9, 4.5) {};

\node[] at (-9, 3) (-9, 3) {};
\node at (-6, 3) (2002) {$[2,0;0,2]$};
\node at (-3, 3) (2001) {$[2,0;0,1]$};
\node at (0, 3) (2101) {$[2,1;0,1]$};
\node at (3, 3) (2100) {$[2,1;0,0]$};
\node at (6, 3) (-1032) {$[-1,0;3,2]$};
\node[] at (9, 3) (9, 3) {};

\node[dot] at (-9, 1.5) (-9, 1.5) {};
\node at (-6, 1.5) (1002) {$[1,0;0,2]$};
\node at (-3, 1.5) (1001) {$[1,0;0,1]$};
\node at (0, 1.5) (1101) {$[1,1;0,1]$};
\node at (3, 1.5) (1100) {$[1,1;0,0]$};
\node at (6, 1.5) (1200) {$[1,2;0,0]$};
\node[dot] at (9, 1.5) (9, 1.5) {};

\node[] at (-9, 0) (-9, 0) {};
\node at (-6, 0) (1012) {$[1,0;1,2]$};
\node at (-3, 0) (1011) {$[1,0;1,1]$};
\node at (0, 0) (1111) {$[1,1;1,1]$};
\node at (3, 0) (1110) {$[1,1;1,0]$};
\node at (6, 0) (1210) {$[1,2;1,0]$};
\node[] at (9, 0) (9, 0) {};

\node[dot] at (-9, -1.5) (-9, -1.5) {};
\node at (-6, -1.5) (0012) {$[0,0;1,2]$};
\node at (-3, -1.5) (0011) {$[0,0;1,1]$};
\node at (0, -1.5) (0111) {$[0,1;1,1]$};
\node at (-1, -1) (C0) {$C_0$};
\node at (1, -1) (s1C0) {$s_1 C_0$};
\node at (-2, -0.5) (s2C0) {$s_2 C_0$};
\node at (-1, -2.0) (s3C0) {$s_3 C_0$};
\node at (3, -1.5) (0110) {$[0,1;1,0]$};
\node at (6, -1.5) (0210) {$[0,2;1,0]$};
\node[dot] at (9, -1.5) (9, -1.5) {};

\node[] at (-9, -3) (-9, -3) {};
\node at (-6, -3) (0022) {$[0,0;2,2]$};
\node at (-3, -3) (0021) {$[0,0;2,1]$};
\node at (0, -3) (0121) {$[0,1;2,1]$};
\node at (3, -3) (0120) {$[0,1;2,0]$};
\node at (6, -3) (0220) {$[0,2;2,0]$};
\node[] at (9, -3) (9, -3) {};

\node[dot] at (-9, -4.5) (-9, -4.5) {};
\node[] at (-6, -4.5) (-6, -4.5) {};
\node[dot] at (-3, -4.5) (-3, -4.5) {};
\node[] at (0, -4.5) (0, -4.5) {};
\node[dot] at (3, -4.5) (3, -4.5) {};
\node[] at (6, -4.5) (6, -4.5) {};
\node[dot] at (9, -4.5) (9, -4.5) {};


\draw[dotted] (-9, 4.5) -- (-6, 4.5);
\draw[dotted] (-6, 4.5) -- (-3, 4.5);
\draw[dotted] (-3, 4.5) -- (0, 4.5);
\draw[dotted] (0, 4.5) -- (3, 4.5);
\draw[dotted] (3, 4.5) -- (6, 4.5);
\draw[dotted] (6, 4.5) -- (9, 4.5);

\draw[dotted] (-9, 3) -- (2002);
\draw[dotted] (2002) -- (2001);
\draw[dotted] (2001) -- (2101);
\draw[dotted] (2101) -- (2100);
\draw[dotted] (2100) -- (-1032);
\draw[dotted] (-1032) -- (9, 3);

\draw[dotted] (-9, 1.5) -- (1002);
\draw[dotted] (1002) -- (1001);
\draw[dotted] (1001) -- (1101);
\draw[dotted] (1101) -- (1100);
\draw[dotted] (1100) -- (1200);
\draw[dotted] (1200) -- (9, 1.5);

\draw[dotted] (-9, 0) -- (1012);
\draw[dotted] (1012) -- (1011);
\draw[dotted] (1011) -- (1111);
\draw[dotted] (1111) -- (1110);
\draw[dotted] (1110) -- (1210);
\draw[dotted] (1210) -- (9, 0);

\draw[dotted] (-9, -1.5) -- (0012);
\draw[dotted] (0012) -- (0011);
\draw[dotted] (0011) -- (0111);
\draw[dotted] (0111) -- (0110);
\draw[dotted] (0110) -- (0210);
\draw[dotted] (0210) -- (9, -1.5);

\draw[dotted] (-9, -3) -- (0022);
\draw[dotted] (0022) -- (0021);
\draw[dotted] (0021) -- (0121);
\draw[dotted] (0121) -- (0120);
\draw[dotted] (0120) -- (0220);
\draw[dotted] (0220) -- (9, -3);

\draw[dotted] (-9, -4.5) -- (-6, -4.5);
\draw[dotted] (-6, -4.5) -- (-3, -4.5);
\draw[dotted] (-3, -4.5) -- (0, -4.5);
\draw[dotted] (0, -4.5) -- (3, -4.5);
\draw[dotted] (3, -4.5) -- (6, -4.5);
\draw[dotted] (6, -4.5) -- (9, -4.5);


\draw[dotted] (-9, 4.5) -- (-9, 3);
\draw[dotted] (-6, 4.5) -- (2002);
\draw[dotted] (-3, 4.5) -- (2001);
\draw[dotted] (0, 4.5) -- (2101);
\draw[dotted] (3, 4.5) -- (2100);
\draw[dotted] (6, 4.5) -- (-1032);
\draw[dotted] (9, 4.5) -- (9, 3);

\draw[dotted] (-9, 3) -- (-9, 1.5);
\draw[dotted] (2002) -- (1002);
\draw[dotted] (2001) -- (1001);
\draw[dotted] (2101) -- (1101);
\draw[dotted] (2100) -- (1100);
\draw[dotted] (-1032) -- (1200);
\draw[dotted] (9, 3) -- (9, 1.5);

\draw[dotted] (-9, 1.5) -- (-9, 0);
\draw[dotted] (1002) -- (1012);
\draw[dotted] (1001) -- (1011);
\draw[dotted] (1101) -- (1111);
\draw[dotted] (1100) -- (1110);
\draw[dotted] (1200) -- (1210);
\draw[dotted] (9, 1.5) -- (9, 0);

\draw[dotted] (-9, 0) -- (-9, -1.5);
\draw[dotted] (1012) -- (0012);
\draw[dotted] (1011) -- (0011);
\draw[dotted] (1111) -- (0111);
\draw[dotted] (1110) -- (0110);
\draw[dotted] (1210) -- (0210);
\draw[dotted] (9, 0) -- (9, -1.5);

\draw[dotted] (-9, -1.5) -- (-9, -3);
\draw[dotted] (0012) -- (0022);
\draw[dotted] (0011) -- (0021);
\draw[dotted] (0111) -- (0121);
\draw[dotted] (0110) -- (0120);
\draw[dotted] (0210) -- (0220);
\draw[dotted] (9, -1.5) -- (9, -3);

\draw[dotted] (-9, -3) -- (-9, -4.5);
\draw[dotted] (0022) -- (-6, -4.5);
\draw[dotted] (0021) -- (-3, -4.5);
\draw[dotted] (0121) -- (0, -4.5);
\draw[dotted] (0120) -- (3, -4.5);
\draw[dotted] (0220) -- (6, -4.5);
\draw[dotted] (9, -3) -- (9, -4.5);


\draw[] (-1032) -- (9, 4.5);
\draw[] (-1032) -- (3, 4.5);
\draw[] (2101) -- (3, 4.5);
\draw[] (2101) -- (-3, 4.5);
\draw[] (2002) -- (-3, 4.5);
\draw[] (2002) -- (-9, 4.5);

\draw[] (-1032) -- (9, 1.5);
\draw[] (-1032) -- (1100);
\draw[] (2101) -- (1100);
\draw[] (2101) -- (1001);
\draw[] (2002) -- (1001);
\draw[] (2002) -- (-9, 1.5);

\draw[] (1012) -- (-9, 1.5);
\draw[] (1012) -- (1001);
\draw[] (1001) -- (1111);
\draw[] (1111) -- (1100);
\draw[] (1100) -- (1210);
\draw[] (9, 1.5) -- (1210);

\draw[] (1012) -- (-9, -1.5);
\draw[] (1012) -- (0011);
\draw[] (0011) -- (1111);
\draw[] (1111) -- (0110);
\draw[] (0110) -- (1210);
\draw[] (1210) -- (9, -1.5);

\draw[] (0022) -- (-9, -1.5);
\draw[] (0022) -- (0011);
\draw[] (0011) -- (0121);
\draw[] (0121) -- (0110);
\draw[] (0110) -- (0220);
\draw[] (9, -1.5) -- (0220);

\draw[] (-9, -4.5) -- (0022);
\draw[] (0022) -- (-3, -4.5);
\draw[] (-3, -4.5) -- (0121);
\draw[] (0121) -- (3, -4.5);
\draw[] (3, -4.5) -- (0220);
\draw[] (0220) -- (9, -4.5);

\end{tikzpicture}
}
\caption{A part of the fundamental apartment $\SS_0$ with the chambers $\Cc_0, s_1 \Cc_0, s_2 \Cc_0$ and $s_3 \Cc_0$ for $n=2$.}
\end{figure}
\end{exa}

\subsection{Self-dual vertices}

For any lattice $\L$ in a symplectic space $(V, \bi-form)$, let us define the dual by
\[
\L^\ast:=\{v \in V \mid \langle v, w\rangle \in \Oc \ \text{for all $w \in \L$}\}.
\]
Note that $\L^\ast$ is also a lattice in $V$.
For every $\a \in F^\times$, we have that $(\a \L)^\ast=\a^{-1}\L^\ast$,
whence the homothety class $[\L^\ast]$ depends only on $[\L]$.
Let us call a vertex $[\L]$ in the building $\Bc_n$ \textit{self-dual} if $[\L^\ast]=[\L]$.
Below we characterize self-dual vertices in terms of labels---it is essentially proved in \cite[Proposition 3.1]{ShemanskeSp}; we give a proof for the sake of completeness.

\begin{lem}\label{Lem:dual}
Fix an integer $n \ge 1$.
For $[\L] \in \Ver(\Bc_n)$,
we have that $[\L^\ast]=[\L]$ if and only if $\lab_n[\L]=0$ or $n \mod 2n$.
\end{lem}

\proof
Fix a symplectic basis of the space over $F$ and identify the space with the standard symplectic space over $F$.
Let $\L_0, \dots, \L_n$ be the sequence of lattices \eqref{Eq:Lk} whose homothety classes form the fundamental chamber $\Cc_0$ in the building $\Bc_n$.
Note that $\L_0$ is primitive and $\L_0^\ast=\L_0$.
For any lattice $\L$, there exist $\g_1, \g_2 \in GL_{2n}(F)$ such that $\L=\g_1 \L_0$ and $\L^\ast=\g_2 \L_0$.
Since $\langle \L^\ast, \L\rangle \subset \Oc$, we have that $\tr^\g_1 \g_2 \in GL_{2n}(\Oc)$,
and thus $\det \tr^\g_1 \g_2 \in \Oc^\times$.
This implies that
\[
\lab_n[\L^\ast]=\ord_\varpi (\det \g_2)=-\ord_\varpi(\det \g_1)=-\lab_n[\L] \mod 2n.
\]
Therefore if $[\L^\ast]=[\L]$, then $2 \lab_n[\L]=0 \mod 2n$, i.e., $\lab_n[\L]=0$ or $n \mod 2n$.
Conversely if $\lab_n[\L]=0$ or $n \mod 2n$,
then $\L=\g\L_0$ or $\g \L_n$ for some $\g \in Sp_n(F)$ since $Sp_n(F)$ acts on $\Bc_n$ transitively on chambers and preserves the labels of vertices.
Noting that $[\L_0^\ast]=[\L_0]$ and $[\L_n^\ast]=[\L_n]$, as well as $\L^\ast=\g \L_0^\ast$ if $\L=\g \L_0$, and $\L^\ast=\g \L_n^\ast$ if $\L=\g\L_n$ for $\g \in Sp_n(F)$,
we conclude that $[\L^\ast]=[\L]$, as required.
\qed

\begin{remark}
If $n \ge 2$, then
for the vertices $[\L] \in \Ver(\Bc_n)$ with $\lab_n[\L] \neq 0, n \mod 2n$,
the homothety class $[\L^\ast]$ does not define a vertex, i.e., $[\L^\ast] \notin \Ver(\Bc_n)$.
Indeed, for the vertex $[\L]$ of label $k \mod 2n$, the homothety class of the dual $[\L^\ast]$ has the label $2n-k \mod 2n$.
For example, if $n=2$, then for the lattice $\L_1$ in \eqref{Eq:Lk}, we have
\[
\L_1^\ast=\Oc \varpi^{-1}u_1\oplus \Oc \varpi^{-1}u_2 \oplus \Oc w_1\oplus \Oc \varpi^{-1}w_2,
\]
and $[\L_1^\ast]$ has the label $1 \mod 2n$, and thus it does not belong to $\Ver(\Bc_2)$.
\end{remark}

\subsection{Special vertices and the special $1$-complex}\label{Sec:special}

For $[\L] \in \Ver(\Bc_n)$,
let us call $[\L]$ a \textit{special vertex} if $[\L^\ast]=[\L]$.
We define the \textit{special $1$-complex} $\Sc_n$ as a $1$-dimensional subcomplex of $\Bc_n$ based on
the set of special vertices
\[
\Ver(\Sc_n):=\big\{[\L] \in \Ver(\Bc_n) \mid [\L^\ast]=[\L]\big\},
\]
and $1$-simplices (edges) are defined between two incident vertices in $\Bc_n$ (cf.\ Section \ref{Sec:symplectic_group}):
for $[\L_1]$, $[\L_2]$ in $\Ver(\Sc_n)$,
we have $[\L_1] \sim [\L_2]$ if and only if there exist representatives $\L_1$ and $\L_2$ from $[\L_1]$ and $[\L_2]$ respectively such that either
$\varpi^{-1}\L_1$ is primitive and $\L_1 \subseteq \L_2 \subseteq \varpi^{-1}\L_1$,
or the analogous relation where the roles of $\L_1$ and $\L_2$ are interchanged holds.
Note that since special vertices are those that are self-dual,
if $\varpi^{-1}\L_1$ is primitive,
then $\L_2/\L_1$ is a maximal totally isotropic subspace of $\varpi^{-1}\L_1/\L_1$ over $\Fb_q$.

Lemma \ref{Lem:dual} shows that $[\L] \in \Ver(\Sc_n)$ if and only if $\lab_n[\L]=0$ or $n \mod 2n$,
and we will see that $\Sc_n$ is connected (Proposition \ref{Prop:connected} below).
Although we do not use it in our main discussion,
it is useful to point out here that $\Sc_n$ admits a structure of bipartite graph.
Namely,
if we decompose the set of vertices into two sets: the one of those with label $0 \mod 2n$ and the other of those with label $n \mod 2n$,
then the two extreme vertices of each edge have distinct labels.

We note that $GSp_n(F)$ does not act on $\Bc_n$ through the linear transformation of lattices.
Indeed, a vertex of label $2n-1 \mod 2n$ in the fundamental chamber $\Cc_0$ is sent by $t_\varpi \in GSp_n(F)$
to a vertex of label $n-1 \mod 2n$, which does not belong to $\Ver(\Bc_n)$.
However, restricted on $\Sc_n$, the group $GSp_n(F)$ acts on $\Sc_n$.
Moreover, the action of $GSp_n(F)$ on $\Sc_n$ is vertex-transitive
since for $t_\varpi=\diag(1, \dots, 1, \varpi, \dots, \varpi)$ in $GSp_n(F)$,
we have that
\[
t_\varpi[\L_0]=[\L_n] \quad \text{where 
$
t_\varpi=\(
\begin{array}{cc}
I_n & 0\\
0 & \varpi I_n
\end{array}
\)
$
and $[\L_0], [\L_n] \in \Cc_0$.
}
\]
Note that $t_\varpi$ permutes the labels on $\Ver(\Sc_n)$.
This defines the action of $PGSp_n(F)$ on $\Sc_n$.
Letting $o:=[\L_0]$, we identify the stabilizer of $o$ in $PGSp_n(F)$ with $K:=PGSp_n(\Oc)$.
If we define
\[
\Sc_n^o:=Sp_n(F)o \quad \text{and} \quad \Sc_n^{o-}:=t_\varpi \Sc_n^o,
\]
then 
\[
\Ver(\Sc_n)=\Sc_n^o\bigsqcup \Sc_n^{o-},
\]
and every edge in $\Sc_n$ has one vertex in $\Sc_n^o$ and the other vertex in $\Sc_n^{o-}$.
The following proposition has been shown by Shemanske; we give a proof for the sake of convenience.

\begin{prop}[Proposition 3.6 in \cite{ShemanskeSp}]\label{Prop:connected}
For every integer $n \ge 1$, the special $1$-complex $\Sc_n$ is connected.
\end{prop}

\proof
Given two special vertices (which are not incident each other), let us take two chambers in such a way that each chamber contains either one or the other vertex.
Since for any two chambers there exists an apartment which contains both of them,
applying an isometry of the building if necessary, we may assume that they are within the fundamental apartment $\SS_0$,
and further one of them is the fundamental chamber $\Cc_0$.
Noting that each reflection in the affine Weyl group maps $\Cc_0$ to an adjacent chamber which shares at least one special vertex with $\Cc_0$.
The other chamber is obtained by a successive application of reflections to $\Cc_0$
and in the resulting sequence of chambers (called a gallery) we find an edge path (consisting of special vertices) connecting the given two special vertices.
This shows that any two special vertices are connected by an edge path in the subcomplex based on the special vertices, i.e., $\Sc_n$ is connected.
\qed

\section{Property (T) and spectral gaps}\label{PT}

\subsection{Property (T)}\label{Sec:PropertyT}

Let $G$ be a topological group
and $(\pi, \Hc)$ be a unitary representation of $G$,
where we assume that any Hilbert space $\Hc$ is complex.  
For any compact subset $Q$ in $G$, 
let
\[
\k(G, Q, \pi):=\inf\Big\{\max_{s \in Q}\|\pi(s)\f-\f\| \mid \f \in \Hc, \ \|\f\|=1\Big\},
\]
and further let
\begin{equation*}
\k(G, Q):=\inf \k(G, Q, \pi),
\end{equation*}
where the above infimum is taken over all equivalence classes of unitary representations $(\pi, \Hc)$ without non-zero invariant vectors.
We call $\k(G, Q)$ the {\it optimal Kazhdan constant} for the pair $(G, Q)$.
We say that $G$ has Property (T) if there exists a compact set $Q$ in $G$ such that $\k(G, Q)>0$.
It is known that for a local field $F$, if $n \ge 2$,
then $Sp_n(F)$ has Property (T),
while if $n=1$, then $Sp_1(F)=SL_2(F)$ and it fails to have Property (T) \cite[Theorem 1.5.3 and Example 1.7.4]{BdlHV}.

For any $n \ge 2$, $PSp_n(F)$ has Property (T) since $Sp_n(F)$ does \cite[Theorem 1.3.4]{BdlHV}.
Similarly, for any $n \ge 2$, the group $PGSp_n(F)$ has Property (T) since $PGSp_n(F)/PSp_n(F)$ admits a finite invariant Borel regular measure (see Section \ref{Sec:symplectic_group} and \cite[Theorem 1.7.1]{BdlHV}).
(We note that for any $n \ge 1$, the group $GSp_n(F)$ does not have Property (T) 
because it admits a surjective homomorphism onto $\Z$ \cite[Corollary 1.3.5]{BdlHV}.)

We say that a subset $Q$ of $G$ is \textit{generating}
if the sub-semigroup generated by $Q$ coincides with $G$.
If $G$ has Property (T) and $Q$ is an arbitrary compact generating set of $G$ (provided that it exists),
then $\k(G, Q)>0$ \cite[Proposition 1.3.2]{BdlHV}.
We will construct an appropriate compact generating set in the following.

\subsection{A random walk operator}\label{Sec:rwop}

In this section, fix an integer $n \ge 1$.
Recall that $K=PGSp_n(\Oc)$,
and letting $o:=[\L_0]$, we identify $K$ with the stabilizer of $o$ in $PGSp_n(F)$.
Let $a:=[t_\varpi] \in PGSp_n(F)$,
and let us choose $\x_i \in PSp_n(F) (\subset PGSp_n(F))$ for $i=0, 1, \dots, n+1$ such that $\x_0:=\id$ and for $i=1. \dots, n+1$ each $\x_i$ projects onto the reflection $s_i$ in the affine Weyl group acting on the fundamental apartment $\SS_0$.

Let us define a subset $\Omega:=\{k \x_i a k', k (\x_i a)^{-1}k' \mid k, k' \in K, i=0, \dots, n+1\}$ in $PGSp_n(F)$,
where we simply write
\[
\Omega= K \Omega_0 K, \quad \text{where $\Omega_0:=\{\x_0 a, \dots, \x_{n+1}a, (\x_0 a)^{-1}, \dots, (\x_{n+1}a)^{-1}\}$}.
\]
Note that $\Omega$ is compact and symmetric, i.e., $x \in \Omega$ if and only if $x^{-1} \in \Omega$.
Let $\n$ be a Haar measure on $K$ normalized so that $\n(K)=1$.
Let us define the probability measure $\m$ on $PGSp_n(F)$
as the distribution of $k \z k'$
where $k, k'$ and $\z$ are independent and $k, k'$ are distributed according to $\n$ and $\z$ is uniformly distributed on
$\{\x_i a, (\x_i a)^{-1} \mid i=0, \dots, n+1\}$. 
In other words,
\[
\m=\n\ast \Unif_{\Omega_0}\ast \n, \quad \text{where $\Unif_{\Omega_0}:=\frac{1}{2(n+2)}\sum_{i=0}^{n+1}\(\d_{\x_i a}+\d_{(\x_i a)^{-1}}\)$},
\]
and $\d_x$ denotes the Dirac distribution at $x$; furthermore the convolution $\m_1 \ast \m_2$ of two probability measures $\m_1, \m_2$ on a group $G$ is defined by
\[
\m_1\ast \m_2(A)=\m_1 \times \m_2\(\{(\g_1, \g_2) \in G \times G \mid \g_1 \g_2 \in A\}\),
\]
for any measurable set $A$ in $G$.
Note that the support of $\m$ is $\Omega$.
For any positive integers $t \ge 1$,
we denote by $\m^{\ast t}$ the $t$-th convolution power of $\m$, 
i.e., $\m^{\ast 1}:=\m$ and $\m^{\ast (t+1)}=\m^{\ast t} \ast \m$ for $t \ge 1$.
If we define the probability measure $\check \m$ on $PGSp_n(F)$ as the distribution of $x^{-1}$ where $x$ has the law $\m$,
then the definition of $\m$ implies that 
\begin{equation}\label{Eq:check}
\check \m=\m.
\end{equation}

\begin{lem}\label{Lem:mu}
We have the following:
\begin{itemize}
\item[(1)] The set $\Omega$ is generating in $PGSp_n(F)$, i.e., $\Omega$ generates $PGSp_n(F)$ as a semigroup.
\item[(2)] Fix an integer $n \ge 1$.
The double coset $K \backslash \Omega /K$ is represented by a finite set $\Omega_0=\{\x_i a, (\x_i a)^{-1}, i=0, \dots, n+1\}$ and
\[
\min_{K \g K \in K \backslash \Omega /K}\m(K \g K)=\frac{1}{2(n+2)}.
\]
Moreover, if $\g$ is distributed according to $\m$ on $PGSp_n(F)$,
then $\g o$ is uniformly distributed on the set of incident vertices to $o=[\L_0]$ in $\Sc_n$.
\end{itemize}
\end{lem}
\proof
Let us show (1).
If we let $K_0:=PSp_n(\Oc)$ and define $\D$ in $K(=PGSp_n(\Oc))$ as the image of $\{t_\lambda \mid \lambda \in \Oc^\times\}$,
then since $K$ contains $K_0$ and $\D$, and
$\Omega$ contains $K\{a, a^{-1}\}K$, the set $\Omega \cdot \Omega$ contains $\bigcup_{i=1}^{n+1}K_0 \x_i K_0$ as well as $K$ (and thus $K_0$ and $\D$).
The group $K_0$ acts on the set of apartments containing $o=[\L_0]$ transitively,
and this implies that
$\bigcup_{i=1}^{n+1}K_0 \x_i K_0$ generates $PSp_n(F)$ as a semigroup, which follows by looking at the induced action of reflections on apartments as in Proposition \ref{Prop:connected}.
Since the quotient $PGSp_n(F)$ modulo $PSp_n(F)$ is generated by the images of $\{a, a^{-1}\}$ and $\D$ (cf.\ Section \ref{Sec:symplectic_group}),
we conclude that $\Omega$ generates $PGSp_n(F)$ as a semigroup.

Let us show (2). The first claim follows since $\Omega=K\Omega_0 K$ and the definition of $\m$ shows that $\m$ yields the uniform distribution on $K \backslash \Omega/K$.
Concerning the second claim, in the fundamental apartment $\SS_0$
we note that $\x_i a o =t_\varpi o$ if $i \neq 1$ and $\x_1 a o=s_1 t_\varpi o$,
and $(\x_i a)^{-1} o=t_\varpi^{-1}o$ if $i \neq n+1$ and $(\x_{n+1} a)^{-1}o=s_\ast t_\varpi o$ where $s_\ast$ is a product of $s_1, s_2, \dots, s_n$ with some repetitions; we note that such an element $s_\ast$ fixes $o$ since it belongs to the spherical Weyl group.
Furthermore $K_0(=PSp_n(\Oc))$ acts on the set of apartments containing $o$ and if we apply $k$ whose distribution is the normalized Haar measure on $K(=PGSp_n(\Oc))$ to an incidence vertex $v$ of $o$, then $k v$ is uniformly distributed on the incident vertices of $o$.
This implies the claim.
\qed

For simplicity of notation, let $G:=PGSp_n(F)$ in the following discussion.
Recall that $\Ver(\Sc_n)=Go$. 
We define the Hilbert space 
\[
\ell^2(\Sc_n):=\Bigg\{\f: \Ver(\Sc_n) \to \C \mid \sum_{v \in \Ver(\Sc_n)}|\f(v)|^2 < \infty\Bigg\},
\]
equipped with the inner product
\[
\lbr \f, \psi\rbr:=\sum_{v \in \Ver(\Sc_n)}\f(v)\wbar{\psi(v)}, \quad \text{for $\f, \psi \in \ell^2(\Sc_n)$}.
\]
Let us define an operator
$\Ac_\m: \ell^2(\Sc_n) \to \ell^2(\Sc_n)$ 
by
\[
\Ac_\m \f(\x o)=\int_{G}\f(\x \g o)d\m(\g) \qquad \text{for $\x \in G$}.
\]
Note that $\Ac_\m$ is well-defined by the definition of $\m$ since $\Ver(\Sc_n)=G o$ and $K$ is the stabilizer of $o$.
Lemma \ref{Lem:mu} (2) shows that $\Ac_\m$ is the normalized adjacency operator on $\Sc_n$.
Since $\check \m=\m$ by \eqref{Eq:check},
the operator $\Ac_\m$ is self-adjoint on $\ell^2(\Sc_n)$.
Similarly if we define $\Ac_{\m^{\ast t}}: \ell^2(\Sc_n) \to \ell^2(\Sc_n)$ for any positive integer $t \ge 1$,
\[
\Ac_{\m^{\ast t}}\f(\x o)=\int_{G}\f(\x \g o)d\m^{\ast t}(\g) \qquad \text{for $\x \in G$},
\]
then we have that by induction
\[
\Ac_{\m}^t=\Ac_{\m^{\ast t}} \quad \text{for all positive integer $t \ge 1$}.
\]

Let us consider any closed subgroup $\G$ of $G$ such that $\G$ acts on $\Sc_n$ from left with a compact quotient space $\G \backslash \Sc_n$, where the action is given by
\[
(\g, v) \mapsto \g v \quad \text{for $\g \in \G$ and $v \in \Sc_n$}.
\]
Since $\G$ acts on $\Sc_n$ by simplicial automorphisms (as $PGSp_n(F)$ does),
the quotient $\G \backslash \Sc_n$ naturally admits a finite (unoriented) graph structure induced from $\Sc_n$.
Let us denote the finite graph by the same symbol $\G \backslash \Sc_n$.
Note that since $\Sc_n$ is connected by Proposition \ref{Prop:connected},
the graph $\G \backslash \Sc_n$ is connected for any such $\G$.
Here, however we do not assume that $\G$ is torsion-free,
thus the graph $\G \backslash \Sc_n$ may have loops and not regular.
Although $\Sc_n$ admits a bipartite graph structure,
$\G \backslash \Sc_n$ is not necessarily bipartite unless $\G$ factors through $PSp_n(F)$.

For each $v \in \Sc_n$, 
let 
\[
\G_v:=\{\g \in \G \mid \g v=v\}.
\]
Note that $\G_v$ is finite; indeed, if $v=\x o$ for $\x \in G$,
then $\x^{-1}\G_v\x$ is in $K$.
Since $\G$ is a discrete subgroup of $G$ and $K$ is compact, $\G_v$ is a finite group.
Since $\G_{\g v}=\g\G_{v}\g^{-1}$ for $\g \in \G$ and $v \in \Sc_n$, whence $|\G_v|$ is independent of the choice of representatives for $v \in \G \backslash \Sc_n$.
Similarly, for $v, w \in \Sc_n$ such that $v$ and $w$ are adjacent in $\Sc_n$,
we define 
\[
\G_{v, w}:=\G_v\cap \G_w.
\]
Considering the diagonal action of $\G$ on $\Sc_n \times \Sc_n$,
we note that $|\G_{v, w}|$ is independent of the choice of representatives for $[v, w]$ in $\G \backslash (\Sc_n\times \Sc_n)$.
Let us define $\ell^2(\G \backslash \Sc_n)$ the space of complex-valued functions on $\G \backslash \Sc_n$ equipped with the inner product defined by
\[
\langle \f, \p\rangle:=\sum_{v \in \G \backslash \Sc_n} \f(v)\wbar{\p(v)}\frac{1}{|\G_{v}|} \quad \text{for $\f, \p \in \ell^2(\G \backslash \Sc_n)$}.
\]
The group $\G$ acts on $\ell^2(\Sc_n)$ by $\f \mapsto \f\circ \g^{-1}$ for $\g \in \G$ and $\f \in \ell^2(\Sc_n)$, and
since this $\G$-action and $\Ac_\m$ on $\ell^2(\Sc_n)$ commute,
the following operator $\Ac_{\G, \m}$ on $\ell^2(\G \backslash \Sc_n)$ is well-defined:
\[
\Ac_{\G, \m}\f(\G \x o)=\int_{G}\f(\G \x \g o)d\m(\g) \qquad \text{for $\G \x o \in \G\backslash \Sc_n$ and $\f \in \ell^2(\G\backslash \Sc_n)$}.
\]
Note that since $\Ac_\m$ defines the simple random walk on $\Sc_n$, i.e., at each step the random walk jumps to a nearest neighbor vertex with equal probability $1/D$ (where $D$ is the degree of $\Sc_n$), the operator $\Ac_{\G, \m}$ defines a random walk (a Markov chain) on $\G \setminus \Sc_n$ with the transition probability
\[
p(\G \xi o, \G \xi \g o)=\sum_{[v, w]}\frac{|\G_v|}{D |\G_{v, w}|},
\]
where the summation runs over all those $[v, w] \in \G \backslash (\Sc_n\times \Sc_n)$ such that $v$ and $w$ are adjacent, $\G v=\G \xi o$ and $\G w =\G \xi \g o$, and we set the probability $0$ if there is no such pair $[v, w]$.
(We recall that $D=N_g(\ell)$ if $F=\Q_\ell$ and $n=g$.)
Since $|\G_{v, w}|=|\G_{w, v}|$ for all $[v, w]$ in $\G \backslash (\Sc_n \times \Sc_n)$,
we have that
\[
\frac{1}{|\G_v|}p(v, w)=\frac{1}{|\G_w|}p(w, v) \qquad \text{for $v, w \in \G \backslash \Sc_n$}, 
\]
the associated random walk on $\G \backslash \Sc_n$ is reversible with respect to the measure $1/|\G_v|$ for each vertex $v$. 
This implies that $\Ac_{\G, \m}$ is self-adjoint, i.e.,
\begin{equation*}
\langle \Ac_{\G, \m}\f, \p\rangle=\langle \f, \Ac_{\G, \m}\p\rangle \qquad \text{for $\f, \p \in \ell^2(\G \backslash \Sc_n)$}.
\end{equation*}
Moreover, $\Ac_{\G, \m^{\ast t}}$ is defined 
by
\[
\Ac_{\G, \m^{\ast t}}\f(\G \x o)=\int_{G}\f(\G \x \g o)d\m^{\ast t}(\g) \qquad \text{for $\G \x o \in \G \backslash \Sc_n$ and $\f \in \ell^2(\G \backslash \Sc_n)$},
\]
and $\Ac_{\G, \m}^t=\Ac_{\G, \m^{\ast t}}$ holds for any positive integer $t \ge 1$.

\subsection{Spectral gap}
We normalize the Haar measure on $G$ in such a way that $K$ has the unit mass.
Let $L^2(\G \backslash G)$ denote the complex $L^2$-space with respect to the (right) Haar measure for which each double coset $\G \x K$ has the mass $1/|\x^{-1}\G \x \cap K|$.
Note that the mass coincides with $1/|\G_{\x o}|$ since $\G_{\x o}=\G \cap \x K \x^{-1}$.
We consider $L^2(\G \backslash G)^K$ the subspace of $K$-fixed vectors in $L^2(\G \backslash G)$ and naturally identify it with
$\ell^2(\G \backslash \Sc_n)$ (including the inner product).
Let us define the unitary representation $\pi$ of $G$ 
on $L^2(\G \backslash G)$ by
\[
\pi(\g)\f(\G \x)=\f(\G \x\g) \quad \text{for $\f \in L^2(\G \backslash G)$ and $\x, \g \in G$}.
\]
Note that $\f \in L^2(\G \backslash G)^K$ if and only if $\pi(k)\f=\f$ for all $k \in K$.

Let
\[
T_\G(\g) \f(\G \x):=\int_K\f(\G \x k \g)\,d\n(k) \quad \text{for $\f \in L^2(\G \backslash G)$ and $\g \in G$},
\]
where we recall that $\n$ is the normalized Haar measure on $K$.

\begin{lem}\label{Lem:Hpi}
For every $n \ge 1$, and for all $\f \in L^2(\G \backslash G)^K$, we have that
\begin{equation*}
\Ac_{\G, \m}\f=\frac{1}{2(n+2)}\sum_{\g \in \Omega_0}T_\G(\g)\f.
\end{equation*}
Moreover, for all $\g \in \G$ and for all $\f_1, \f_2 \in L^2(\G\backslash G)^K$, we have that
\[
\langle T_\G(\g)\f_1, \f_2\rangle=\langle \pi(\g)\f_1, \f_2\rangle.
\]
\end{lem}
\proof
First let us show the first claim.
Recalling that $\m=\n\ast \Unif_{\Omega_0}\ast \n$,
for $\f \in L^2(\G \backslash G)^K$ and $\x, \g \in G$, we have that 
\begin{align*}
\Ac_{\G, \m}\f(\G \x)	&=\int_G \f(\G \x\g)\,d\m(\g)\\	
									&=\int_{K \times \Omega_0 \times K}\f(\G \x k_1 \g k_2)\,d\n(k_1)d\Unif_{\Omega_0}(\g)d\n(k_2)\\
									&=\frac{1}{2(n+2)}\sum_{\g \in \Omega_0}\int_K \f(\G \x k\g)\,d\n(k)
									=\frac{1}{2(n+2)}\sum_{\g \in \Omega_0}T_\G(\g)\f(\G \x),
\end{align*}
where the third equality follows since $\f$ is a $K$-fixed vector and the last identity follows from the definition of $T_\G(\g)$.
Hence the first claim holds.

Next let us show the second claim.
If we denote the right-invariant Haar measure on $\G\backslash G$ by $m_{\G \backslash G}$,
then
\begin{align*}
\langle T_\G(\g)\f_1, \f_2\rangle	&=\int_{\G \backslash G}\(\int_K\f_1(\G \x k\g)\,d\n(k)\)\wbar{\f_2(\G \x)}\,dm_{\G \backslash G}(\x)\\
													&=\int_K \int_{\G \backslash G}\f_1(\G \x k\g)\wbar{\f_2(\G \x)}\,dm_{\G \backslash G}(\x)\,d\n(k)\\
													&=\int_K \int_{\G \backslash G}\f_1(\G \x\g)\wbar{\f_2(\G \x)}\,dm_{\G \backslash G}(\x)\,d\n(k)\\
													&=\int_{\G \backslash G}\f_1(\G \x\g)\wbar{\f_2(\G \x)}\,dm_{\G \backslash G}(\x)
													=\langle \pi(\g)\f_1, \f_2\rangle,
\end{align*}
where the second equality follows by the Fubini theorem and the third equality holds under the change of variables $\G \x \mapsto \G \x k$ since $m_{\G \backslash G}$ is right-invariant, $\f_2$ is a $K$-fixed vector, and $\n$ is normalized so that $\n(K)=1$.
We conclude the second claim.
\qed

Let us define 
\[
\ell_0^2(\G \backslash \Sc_n):=\Bigg\{\f \in \ell^2(\G \backslash \Sc_n) \mid \sum_{\G v \in \G \backslash \Sc_n}\f(\G v)\frac{1}{|\G_v|}=0\Bigg\},
\]
i.e., $\ell_0^2(\G \backslash \Sc_n)$ is the orthogonal complement to the space of constant functions in $\ell^2(\G \backslash \Sc_n)$.
Note that $\Ac_{\G, \m}$ acts on $\ell_0^2(\G \backslash \Sc_n)$ since the operator is self-adjoint.
 
Given the right representation $(\pi, L^2(\G \backslash G))$, 
letting $L_0^2(\G \backslash G)$ be the orthogonal complement to constant functions in $L^2(\G \backslash G)$,
we define $(\pi_0, L_0^2(\G \backslash G))$ by restricting $\pi$ to $L_0^2(\G \backslash G)$.
The space $\ell_0^2(\G \backslash \Sc_n)$ is identified with the space of $K$-fixed vectors in $L_0^2(\G \backslash G)$ under the identification between $\ell^2(\G \backslash \Sc_n)$ and $L^2(\G \backslash G)^K$.
It is crucial that $\pi_0$ has no non-zero invariant vector.

\begin{prop}\label{Prop:kappa}
For every $n \ge 1$, let $\G$ be a closed subgroup of $G=PGSp_n(F)$ such that $\G \backslash \Sc_n$ is finite.
For all $\f  \in \ell_0^2(\G\backslash \Sc_n)$ with $\|\f\|=1$,
we have that
\begin{equation*}
\langle (I-\Ac_{\G, \m})\f, \f\rangle \ge \frac{1}{4(n+2)}\k(G, \Omega)^2,
\end{equation*}
where $\k(G, \Omega)$ is the optimal Kazhdan constant for the pair $(G, \Omega)$.
\end{prop}
\proof
For $\f \in \ell_0^2(\G\backslash \Sc_n)$, 
it follows that
\begin{align*}
\langle (I-\Ac_{\G, \m})\f, \f\rangle
&=\langle \f, \f\rangle-\frac{1}{2(n+2)}\sum_{\g \in \Omega_0}\langle T_\G(\g)\f, \f\rangle \\
&=\langle \f, \f \rangle -\frac{1}{2(n+2)}\sum_{\g \in \Omega_0}\langle \pi(\g)\f, \f\rangle
=\frac{1}{4(n+2)}\sum_{\g \in \Omega_0}\|\f-\pi_0(\g)\f\|^2,
\end{align*}
where identifying $\f$ with a $K$-fixed vector, we have used Lemma \ref{Lem:Hpi} in the first and second lines, and the last equality follows since $\pi_0$ is the restriction of $\pi$ and
\[
\|\f-\pi_0(\g)\f\|^2=\langle \f, \f\rangle-\langle \pi_0(\g)\f, \f\rangle-\langle \pi_0(\g^{-1})\f, \f \rangle+\langle \pi_0(\g)\f, \pi_0(\g)\f\rangle,
\]
and $\pi_0(\g)$ is unitary, and furthermore $\g \in \Omega_0$ if and only if $\g^{-1} \in \Omega_0$.
Moreover, we have that
\begin{align*}
\sum_{\g \in \Omega_0}\|\f-\pi_0(\g)\f\|^2 \ge \max_{\g \in \Omega_0}\|\f-\pi_0(\g)\f\|^2
=\max_{\g \in \Omega}\|\f-\pi_0(\g)\f\|^2,
\end{align*}
which follows from the first claim of Lemma \ref{Lem:mu} (2) and since $\f$ is a $K$-fixed vector and $\pi_0$ is a unitary representation.
Therefore we obtain
\[
\langle (I-\Ac_{\G, \m})\f, \f\rangle \ge \frac{1}{4(n+2)}\max_{\g \in \Omega}\|\f-\pi_0(\g)\f\|^2.
\]
Since $\pi_0$ has no non-zero invariant vector,
we conclude the claim.
\qed

\begin{thm}\label{thm:gap}
If we fix an integer $n \ge 2$, then there exists a positive constant $c_n>0$
such that for any closed subgroup $\G$ in $PGSp_n(F)$ with finite quotient $\G \backslash \Sc_n$,
we have
\[
\lambda_2(\D_{\G,\mu}) \ge c_n,
\]
where $\D_{\G, \mu}=I-\Ac_{\G, \mu}$.
\end{thm}

\proof
Since we have that
\[
\lambda_2(\D_{\G, \mu})=\inf_{\f \in \ell^2_0(\G \backslash \Sc_n), \ \|\f\|=1}\langle (I-\Ac_{\G, \m})\f, \f\rangle,
\]
Proposition \ref{Prop:kappa} implies that
\[
\lambda_2(\D_{\G, \mu}) \ge \frac{1}{4(n+2)}\k(G, \Omega)^2.
\]
Furthermore, $\k(G, \Omega)>0$ since $G=PGSp_n(F)$ has Property (T) if $n \ge 2$ and $\Omega$ is a compact generating set of $G$ by Lemma \ref{Lem:mu} (1) (cf.\ Section \ref{Sec:PropertyT}).
Letting 
\begin{equation}\label{Eq:cn}
c_n:=\frac{1}{4(n+2)}\k(G, \Omega)^2,
\end{equation}
we obtain the claim.
\qed

The proof of Theorem \ref{main} now follows from Theorem \ref{thm:gap} with $\G$ applied to $G_g(\Z[1/\ell])$ modulo the center and Corollary \ref{lwm}.

\subsection{An explicit lower bound for the spectral gap}\label{Sec:Oh}
Appealing to the results by Oh \cite{Oh},
we obtain explicit lower bounds for the second smallest eigenvalues of Laplacians on the graphs $\Gc^{SS}_{g}(\ell,p)$ for $g \ge 2$.
\begin{cor}\label{Cor:Oh}
For every integer $g \ge 2$, for all primes $\ell$ and $p$ with $p \neq \ell$,
\[
\lambda_2\(\Gc^{SS}_{g}(\ell,p)\) \ge \frac{1}{4(g+2)}\(\frac{\ell-1}{2(\ell-1)+3\sqrt{2\ell (\ell+1)}}\)^2.
\]
\end{cor}

\proof
We keep the notations in the preceding subsections and put $n=g$.
Let $F:=\Q_\ell$.
Note that $\Omega^2$ contains $K$ and $a^2$.
The definition of the optimal Kazhdan constant shows that
\[
\k\(G, \Omega^2\) \ge \k\(PSp_n(\Q_\ell), \Omega^2\cap PSp_n(\Q_\ell)\).
\]
Furthermore the right hand side is at least 
$\k\(Sp_n(\Q_\ell), \Omega_\ast\)$,
where
\[
\Omega_\ast:=\{Sp_n(\Z_\ell), s\} \quad \text{and} \quad s:=\diag(\ell^{-1}, \dots, \ell^{-1}; \ell, \dots, \ell).
\]
Applying \cite[Theorem 8.4]{Oh} to $Sp_n(\Q_\ell)$ for $n\ge 2$ with a maximal strongly orthogonal system ${\rm L}$ in the case of $C_n$ $(n \ge 2)$ \cite[Appendix]{Oh},
we have that
\[
\k\(Sp_n(\Q_\ell), \Omega_\ast\) \ge \chi_{\rm L}(s)=\frac{\sqrt{2(1-\xi_{\rm L}(s))}}{\sqrt{2(1-\xi_{\rm L}(s))}+3},
\]
where
\[
\xi_{\rm L}(s) \le \frac{2(\ell-1)+(\ell+1)}{\ell(\ell+1)}=\frac{3\ell-1}{\ell(\ell+1)}.
\]
Hence we have for all $n \ge 2$ and all prime $\ell$,
\[
\k\(Sp_n(\Q_\ell), \Omega_\ast\) \ge \frac{\sqrt{2}(\ell-1)}{\sqrt{2}(\ell-1)+3\sqrt{\ell(\ell+1)}},
\]
and since $\k(G, \Omega) \ge (1/2)\k(G, \Omega^2)$,
we obtain
\[
\k(G, \Omega) \ge \frac{\ell-1}{2(\ell-1)+3\sqrt{2\ell (\ell+1)}}.
\]
Combining the above inequality with \eqref{Eq:cn} in the proof of Theorem \ref{thm:gap}, we conclude that for all $n \ge 2$ and all prime $\ell$,
\[
\lambda_2(\D_{\G, \m}) \ge \frac{1}{4(n+2)}\(\frac{\ell-1}{2(\ell-1)+3\sqrt{2\ell (\ell+1)}}\)^2.
\]
Applying to the case when $\G$ is $G_g(\Z[1/\ell])$ modulo the center together with Corollary \ref{lwm} yields the claim.
\qed

\section{Some remarks on Algebraic modular forms for $GUSp_g$}\label{GUSp}
In this section, we study algebraic modular forms on 
$G_g(\A_\Q)=GUSp_g(\A_\Q)$ which can be also regarded as functions on $SS_g(p)$. 
When $g=1$, Pizer applied the Jacquet-Langlands correspondence to study 
${\rm Gr}_1(\ell,p)$ \cite{Pizer-graph} and 
he showed such graphs are Ramanujan. 

However, for $g\ge 2$, 
the Jacquet-Langlands correspondence between $G_g$ and $GSp_g$ is not 
still fully understood well though in the case when $g=2$, there are several important works which have recently come out (see \cite{Hoften},\cite{RW}). 

It seems morally possible to classify algebraic modular forms on $G_g$ by using 
the trace formula approach as in \cite{RW} and relate them to Siegel 
modular forms on $GSp_g$ though we need to prove the transfer theorem for 
Hecke operators with respect to the principal genus. 
Then Arthur's endoscopic classification (cf. \cite{Arthur1},\cite{Arthur2}) for $GSp_g$ which is not still established except for $g\le 2$ 
would be used to obtain desired results for $\mathcal{G}^{SS}_{2}(\ell, p)$.  
From this picture, it would be easy for experts in the theory of automorphic representations 
to guess the upper bounds of Satake-parameters at $\ell$ for Hecke eigen algebraic modular forms and Hecke eigen Siegel modular forms as well. 
It should be remarked that there are some classes of Hecke eigen Siegel modular forms 
which does not satisfy Ramanujan conjecture. They are so called CAP forms 
(cf. Section 3.9 of  \cite{Gan}). 
However, such forms are expected to be negligible among all forms when $p$ goes to infinity 
and this is in fact true for Siegel modular forms on $GSp_g$ (see \cite{KWY},\cite{KWY1}). 

In fact, the third author showed that in fact, it is also true for $M(K)$ 
when $g=2$ \cite{Yamauchi}. 
With this background from the theory of automorphic representations, in this section,  
we propose a conjecture that 
$\mathcal{G}^{SS}_{g}(\ell, p)$ is asymptotically relatively Ramanujan when $p$ goes to infinity (see Definition \ref{ARR}). 
We also give a conjecture related to Conjecture 1 of \cite{FS2} in our setting.

Henceforth, we use the index $n$ to stand for $G_n$ instead of the index $g$ of 
$G_g$ to avoid 
the confusion in which we use $g$ as an element of the groups.

\subsection{Gross's definition}\label{GrossDef}
We refer Chapter II of \cite{Gross} for the notation and basic facts. 
Recall the notation in Subsection \ref{CN}.  
Put $K=K(\O^n)$. Since $B$ is definite, $G_n(\R)=GUSp_n(\R)$ is compact modulo its center. It follows from  (\ref{desc1}) that 
\begin{equation}\label{strong-approx} 
G_n(\Q)\bs G_n(\A_\Q)/(K\times G_n(\R)^+)=G_n(\Q)\bs G_n(\A_f)/K
\end{equation}
where $G_n(\R)^+$ stands for the connected component of the identity element 
and the cardinality of (\ref{strong-approx}) is nothing but the class number $H_n(p,1)$ of the principal genus. 
According to Chapter II-4 of \cite{Gross}, we define 
the $\C$-vector space $M(K)$ consisting of all locally constant functions $f:G_n(\A_\Q)\lra \C$ such that 
$$f(\gamma g k g_\infty)=f(g),\ g\in G_n(\A_\Q)$$
for all $\gamma\in G_n(\Q),\ k\in K$, and $g_\infty\in G_n(\R)^+$. 
Put $h:=H_n(p,1)$ and pick $\{\gamma_i\}_{i=1}^h$ with $\gamma_i\in G_n(\A_f)$ a complete system of the 
representatives of (\ref{strong-approx}). 
By definition, the space $M(K)$ is generated by 
the characteristic functions $\varphi_i,\ 1\le i\le h$ of 
$G_n(\Q)\gamma_i K$. Hence we have 
$M(K)\simeq \C^{\oplus h}$. 
We define a hermitian inner product $( \ast,\ast )_K$ on $M(K)$ by 
\begin{equation}\label{pairing-alg}( f_1,f_2 )_K:=\sum_{\gamma\in G_n(\Q)\bs G_n(\A_f)/K}
f_1(\gamma)\overline{f_2(\gamma)}\frac{1}{|{\rm Aut}(\gamma)|}
\end{equation}
for $f_1,f_2\in M(K)$ where ${\rm Aut}(\gamma):=
(G_n(\Q)\cap \gamma  K \gamma^{-1})Z(\A_f)/Z(\A_f)$.
Let $\varphi$ be a non-zero constant function on $G_n(\A_\Q)$.  
We denote by $M_0(K)$ the orthogonal complement of $\C\varphi$ in $M(K)$. 
Clearly, ${\rm dim}(M_0(K))=h-1=H_n(p,1)-1$. 
\begin{Def}\label{AMF}\upshape{
Each element of $M(K)$ is said to be an algebraic modular form on $G_n(\A_\Q)=GUSp_n(\A_\Q)$ of weight zero with level $K$. }
\end{Def}
For each prime $\ell\neq p$ we define the (unramified) Hecke algebra 
$$\mathcal{H}_\ell=\C[G_n(\Z_\ell)\bs G_n(\Q_\ell)/G_n(\Z_\ell)]\simeq 
\C[GSp_n(\Z_\ell)\bs GSp_n(\Q_\ell)/GSp_n(\Z_\ell)]$$ at $\ell$ 
which is generated by the characteristic functions of form 
$G_n(\Z_\ell)gG_n(\Z_\ell)$ for $g\in G_n(\Q_\ell)$. 
Let $e_\ell$ be the characteristic function of $G_n(\Z_\ell)$ which is the identity element of 
$\mathcal{H}_\ell$. 
Let $\mathbb{T}^{(p)}=\otimes'_{\ell\neq p}\mathcal{H}_\ell$ be the 
restricted tensor product of $\{\mathcal{H}_\ell\}_{\ell\neq p}$ with respect to the 
identity elements $\{e_\ell\}_{\ell\neq p}$. 
We call $\mathbb{T}^{(p)}$ the Hecke ring outside $p$ and 
it is well-known that $\mathbb{T}^{(p)}$ acts on $M(K)$ and also on $M_0(K)$ (cf. Section 6 of \cite{Gross}). 
\begin{Def}\label{HEAMF}\upshape{
Each element of $M(K)$ is said to be a Hecke eigenform outside $p$ if 
it is a simultaneous eigenform for all elements in $\mathbb{T}^{(p)}$. }
\end{Def}
By using the Hermitian paring (\ref{pairing-alg}), 
we can check that there exists an orthonormal basis $HE(K)$ of $M_0(K)$ which 
consists of Hecke eigenforms outside $p$. 
For each non-zero $F$ in $HE(K)$ and an element $T\in \mathbb{T}^{(p)}$, 
we denote by $\lambda_F(T)$ the eigenvalue of $F$ for $T$. 
Since $F$ has the trivial central character, $\lambda_T(F)$ is a real number. 
Recall the Hecke operator $T(\ell)$ in Section \ref{hecke-at-ell}. 
By definition $T(\ell)$ is the characteristic function of 
$G_n(\Z_\ell)t_\ell G_n(\Z_\ell)$ where 
$t_\ell=\diag(\overbrace{1,\ldots,1}^{n},\overbrace{\ell,\ldots,\ell}^{n})$. 

As explained at the beginning of this section, under the background of the theory of automorphic representations, 
there are CAP forms in $HE(K)$ which do not satisfy Ramanujan conjecture but 
they are expected to be negligible among all forms when $p$ goes to infinity. 
In this vein, we propose the following:
\begin{conj}\label{asymp-rel-ramanujan}
Put $d_{n,p}:={\rm dim}M_0(K)=|HE(K)|=H_n(p,1)-1$. 
For each $\ell\neq p$, it holds that
$$\limsup_{p\to \infty}\frac{1}{d_{n,p}}\sum_{F\in HE(K)}|\lambda_F(T(\ell))|
\le 2^n \ell^{\frac{n(n+1)}{4}}.$$
\end{conj}
The bound is nothing but the Ramanujan bound for $T(\ell)$ for 
Siegel cusp forms on $GSp_n$ whose automorphic representations 
are tempered at $\ell$ (see Section 19 of \cite{vdGeer}). 
It also coincides with the spectral radius of the special $1$-complex $\Sc_n$ (see Proposition 2.6 of 
\cite{Setyadi}).

\subsection{A speculation for bounds of eigenvalues of $T(\ell)$} \label{eigenvalue}
Let us consider the case when $n=2$. Then we have three types of 
CAP forms in $M(K)$ which are given by 
\begin{enumerate}
\item cuspidal forms associated to Borel subgroup;
\item cuspidal forms associated to Klingen parabolic subgroup;
\item cuspidal forms associated to Siegel parabolic subgroup. 
\end{enumerate}
For the third case, historically, they are also called Saito-Kurokawa lifts \cite{Gan}. 
Any form in $M(K)$ has the trivial central character and this shows the first case 
occurs only for the constant function. The second case also can not occur 
since such a form has a non-trivial central character. We remark that the eigenvalue of the constant function for $T(\ell)$ is 
$\ell^3+\ell^2+\ell+1=N_2(\ell)$.  

For the third case, the eigenvalue $\lambda_{F_{\tiny{{\rm Siegel}}}}(T(\ell))$ for each 
cuspidal form $F_{\tiny{{\rm Siegel}}}$ associated to Siegel parabolic subgroup satisfies 
\begin{equation}\label{CAP-Siegel}
\ell^2+1-2\ell\sqrt{\ell}\le \lambda_{F_{\tiny{{\rm Siegel}}}}(T(\ell))\le \ell^2+1+2\ell\sqrt{\ell}.
\end{equation}
As noticed before, $\lambda_F(T(\ell))$ is always a real number for each $F\in M(K)$ 
since $F$ has a trivial central character. 

For each non-CAP form $F$ in $M_0(K)$ we would expect that 
\begin{equation}\label{non-CAP}
|\lambda_F(T(\ell))|\le 4\ell\sqrt{\ell}
\end{equation}
and non-CAP forms are majority of $M_0(K)$. 

It is easy to see that $\ell^2+1+2\ell\sqrt{\ell}$ is the maximum among  the upper bounds of 
(\ref{CAP-Siegel}) and (\ref{non-CAP}) when $\ell\ge 5$.  
Let $1=\mu_1>\mu_2\ge \cdots \ge \mu_m>-1$ be the eigenvalues of the 
random walk matrix (the normalized adjacency matrix) for $\mathcal{G}^{SS}_{2}(\ell, p)$ with 
$m=|SS_2(p)|$ and 
put $\lambda_i=1-\mu_i$.  
\begin{conj}\label{refinement}
Assume $p\ge 5$. For each prime $\ell\neq p$, it holds that 
$$1-\frac{\max\{4\ell\sqrt{\ell},\ell^2+1+2\ell\sqrt{\ell}\}}{N_{2}(\ell)}\le \lambda_i \le 
1+\frac{4\ell\sqrt{\ell}}{N_{2}(\ell)}
$$
In particular, when $\ell=2$, 
$$1-\frac{8 \sqrt{2}}{15}=0.24575...  \le \la_2, \quad 2-\la_m \le 
1+\frac{8 \sqrt{2}}{15}=1.75425....$$
\end{conj}
\begin{remark}\label{compare}
Comparing with Conjecture 1 of \cite{FS2}, 
let
$$\lambda_{\star}(\mathcal{G}^{SS}_{2}(\ell, p))=
\min\{\la_2,2-\la_m\}.$$ In particular, when $\ell=2$, it yields 
$$1-\frac{8 \sqrt{2}}{15}=0.24575...  \le \lambda_{\star}(\mathcal{G}^{SS}_{2}(2, p)) \le 
1+\frac{8 \sqrt{2}}{15}=1.75425....$$
Further, we would be able to check the lower and upper bounds would be sharp  by using 
the classification of Saito-Kurokawa forms due to Gan \cite{Gan} and 
equidistribution for Satake parameters of newforms in $S_4(\Gamma_0(p))$ when 
$p$ goes to infinity. Here $S_4(\Gamma_0(p))$ stands for the space of 
elliptic cusp forms of weight $4$ with respect to $\Gamma_0(p)\subset SL_2(\Z)$. 
The assumption $p\ge 5$ in Conjecture \ref{refinement} is used to guarantee $S_4(\Gamma_0(p))\neq \{0\}$. 
\end{remark}

\subsection{Not being Ramanujan is not necessary fared}
As is expected naturally for experts in the theory of automorphic representations, 
the eigenvalues of the adjacency matrix for $\mathcal{G}^{SS}_{g}(\ell, p)$ 
do not satisfy the Ramanujan bound when $g\ge 2$. 
However, in view of the theory of automorphic forms, it is plausible because of 
the existence of CAP forms violating Ramanujan property. 
Even one can prove, in fact, that it happens for 
$\mathcal{G}^{SS}_{2}(\ell, p)$ for 
each $\ell\ge 17$ and $p\ge 5$ by using the results in \cite{Gan}. 
In Section 10.1 of \cite{JZ}, they gave an example satisfying the Ramanujan bound but this is just 
possible only for small $\ell$ (less than 13 to be precise). 
These things would happen similarly for general $g$. 
Nonetheless, $\mathcal{G}^{SS}_{g}(\ell, p)$  has a nice property as Theorem \ref{main} 
speaks out.  

Therefore, a more conceptual, finer notation should be introduced to 
measure how good  a family of regular graphs is.  
We here propose the following.
Let $\{X_i\}_{i\in I}$ be a family of $d$-regular graphs indexed by 
an ordered set $I$ such that $\ds\lim_{i\to \infty}|X_i|=\infty$. 
Let $\lambda_{{\rm mean}}(X_i)$ be the average of the absolute values of all 
eigenvalues of the normalized adjacency matrix for $X_i$. 
Suppose there exist a prime $\ell$ and a reductive algebraic group $G$ over $\Z$  
such that for each $i\in I$, 
$X_i=\Gamma_i\backslash G(\Q_\ell)/Z_G(\Q_\ell)G(\Z_\ell)$ for some 
lattice $\Gamma_i$ in $G(\Q_\ell)$ for each $i\in I$. Here $Z_G$ is the center of $G$. 
\begin{Def}\label{ARR}
We say $\{X_i\}_{i\in I}$ is asymptotically relatively Ramanujan if 
$$\limsup_{i\to\infty}\lambda_{{\rm mean}}(X_i)\le 
\rho(\mathcal{B}^1(G^{{\rm der}}))$$
where $\mathcal{B}^1(G^{{\rm der}})$ is 
a subgraph of the 1-skelton of the building $\mathcal{B}(G)$ for  
$G(\Q_\ell)/Z_G(\Q_\ell)G(\Z_\ell)$ such that $G(\Q_\ell)$ acts transitively on 
$\mathcal{B}^1(G^{{\rm der}})$ and $\rho(\mathcal{B}^1(G^{{\rm der}}))$ stands for 
the spectral radius of the graph.  
\end{Def}
Our graph $\mathcal{G}^{SS}_{g}(\ell, p)$ is related to $G=GSp_g$ with  $G^{{\rm der}}=Sp_g$ and its  
spectral radius is computed in Proposition 2.6 of \cite{Setyadi} as already mentioned.

\bibliographystyle{alpha}

\end{document}